\documentclass{article}

\usepackage{graphicx}
\usepackage{amsthm,xcolor}
\usepackage{amsfonts}
\usepackage{amsmath}
\usepackage{amscd}
\usepackage{amssymb}
\usepackage{alltt}
\usepackage{url}
\usepackage{ellipsis}
\usepackage{tikz} %
\usetikzlibrary{chains,shapes,arrows,%
 trees,matrix,positioning,decorations,fadings}

\def\tikzfig#1#2#3{%
\begin{figure}[htb]%
  \centering
\begin{tikzpicture}#3
\end{tikzpicture}
  \caption{#2}
  \label{fig:#1}%
\end{figure}%
}
\def\smalldot#1{\draw[fill=black] (#1) %
 node [inner sep=1.3pt,shape=circle,fill=black] {}}
\newtheorem{theorem}{Theorem}[subsection]
\newtheorem{lemma}[theorem]{Lemma}

\newtheorem{working}[theorem]{Working Hypothesis}

\theoremstyle{remark}
\newtheorem{remark}[equation]{Remark}

\newcommand{\ring}[1]{\mathbb{#1}}
\newcommand{\op}[1]{\hbox{#1}}

\newcommand{\ang}[1]{\left\langle{#1}\right\rangle}

\def\sl{\mathfrak{sl}_2(\ring{R})}
\def\SL{\op{SL}_2(\ring{R})}
\def\SO{\op{SO}_2(\ring{R})}
\def\h{\mathfrak h}
\def\hstar{{\mathfrak h}^\star}
\def\Mstar{M^\star}
\def\D{\ring{D}}

\def\Hlie{H_{Lie}}

\def\DR{D_{min}}

\newcommand\Lsing{\Lambda_{sing}}
\newcommand\lsing{\lambda_{sing}}
\newcommand\ee[1]{e_{#1}^*}
\newcommand{\partials}[2]{\frac{\partial #1}{\partial #2}}

\title{The Reinhardt Conjecture as an Optimal Control Problem}
\author{Thomas C. Hales\thanks{Research supported by NSF grant
    1104102.  I thank W\"oden Kusner for discussions related to this
    problem.}}  
\date{} 

\begin{document}

\maketitle

\begin{abstract} 
  In 1934, Reinhardt conjectured that the shape of the centrally
  symmetric convex body in the plane whose densest lattice packing has
  the smallest density is a smoothed octagon.  This conjecture is
  still open.  We formulate the
  Reinhardt Conjecture as a problem in optimal control theory.

  The smoothed octagon is a Pontryagin extremal trajectory with
  bang-bang control.  More generally, the smoothed regular $6k+2$-gon
  is a Pontryagin extremal with bang-bang control.  The smoothed
  octagon is a strict (micro) local minimum to the optimal control
  problem.
  
  The optimal solution to the Reinhardt problem is a trajectory
  without singular arcs.  The extremal trajectories that do not meet
  the singular locus have bang-bang controls with finitely many
  switching times.

  Finally, we reduce the Reinhardt problem to an optimization problem
  on a five-dimensional manifold.  (Each point on the manifold is an
  initial condition for a potential Pontryagin extremal lifted
  trajectory.)  We suggest that the Reinhardt conjecture might
  eventually be fully resolved through optimal control theory.

  Some proofs are computer-assisted using a computer algebra system.
\end{abstract}

\parskip=0.8\baselineskip
\baselineskip=1.05\baselineskip

\newenvironment{blockquote}{%
  \par%
  \medskip%
  \baselineskip=0.7\baselineskip%
  \leftskip=2em\rightskip=2em%
  \noindent\ignorespaces}{%
  \par\medskip}

\section{Introduction}

In 1934, Reinhardt conjectured that the shape of centrally symmetric
body in the plane whose densest lattice packing has the smallest
density is a smoothed octagon (Figure \ref{fig:octagon}).  The
corners of the octagon are rounded by hyperbolic arcs.  For popular
accounts of the Reinhardt conjecture, including some spectacular
animated graphics by Greg Egan, see \cite{baez-egan}, \cite{baez}.

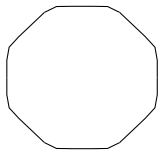
\tikzfig{octagon}{ A smoothed octagon is conjectured to have the worst
  best packing among centrally symmetric disks in the plane.  }{
\draw (1, 0) --  (1., 0.229378) --  (0.969314, 0.403215) --  (0.839864, 
  0.544075) --  (0.659754, 0.714692) --  (0.5, 0.866025);
\draw (0.5, 0.866025) --  (0.340246, 0.944071) --  (0.129449, 
  0.94729) --  (-0.129449, 0.94729) --  (-0.340246, 0.944071) --  (-0.5, 
  0.866025);
\draw (-0.5, 0.866025) --  (-0.659754, 0.714692) --  (-0.839864, 
  0.544075) --  (-0.969314, 0.403215) --  (-1., 0.229378) --  (-1., 0);
\draw (-1., 0) --  (-1., -0.229378) --  (-0.969314, -0.403215) --  (-0.839864,  
-0.544075) --  (-0.659754, -0.714692) --  (-0.5, -0.866025);
\draw (-0.5, -0.866025) --  (-0.340246, -0.944071) --  (-0.129449, -0.94729) --   
(0.129449, -0.94729) --  (0.340246, -0.944071) --  (0.5, -0.866025);
\draw (0.5, -0.866025) --  (0.659754, -0.714692) --  (0.839864, -0.544075) --   
(0.969314, -0.403215) --  (1., -0.229378) --  (1., 0);
}

This article is a continuation of an article from 2011, which formulates
the Reinhardt conjecture as a problem in the calculus of variations
\cite{hales2011reinhardt}.  This article reformulates the Reinhardt
conjecture as an optimal control problem.

Bang-bang controls of an optimal control problem are controls that
switch between extreme points of a convex control set (often with a
finite number of switches).  A major theme of optimal control is the
study of bang-bang controls, and the extremal trajectories of many
control problems have bang-bang controls.  Intuitively, bang-bang
controls switch from one extreme position to another: navigating a
craft by flooring the accelerator pedal then slamming on the brakes;
or steering a vehicle by making the sharpest possible turns to the
left and to the right; or maximizing wealth by investing all resources
in a single financial asset for a time, then suddenly moving all
resources elsewhere.

The basic insight of this article, from which everything else follows,
is that the smoothed octagon can be described by a bang-bang control
with finitely many switches.  The smoothed octagon exhibits the
extreme behavior that is characteristic of a bang-bang control: each
arc of the smoothed octagon is flattened out as much as possible (the
straight edges) or is as highly curved as possible (the hyperbolic
arcs), with finitely many switches between these extremes.  Because of
this bang-bang behavior, the natural context for the Reinhardt
conjecture is optimal control theory.  Viewed in this context, the
Reinhardt conjecture is transformed from a puzzling problem in
discrete geometry to a rather typical problem in optimal control.  In
fact in many ways, this is a textbook example of optimal control, by
embodying significant aspects of the general theory in a single
problem.

The original and guiding inspiration for this research was the visual
similarity between the solutions to the Dubins car problem and
segments of smoothed polygons (Figure \ref{fig:dubins}).  Recall that
the Dubins car problem is the optimal control problem that asks for
the shortest path in the plane from an initial position (and direction) to a
terminal position and direction, subject to a given bound on the
absolute value of the curvature at each point of the path.  Roughly
speaking, the Reinhardt problem is a modification of the Dubins
problem that imposes hexagonal symmetry and a steering wheel that
turns only to the left.  In both cases, curvature constraints force
the (conjectural) solution to consist of finitely many straight
segments and arcs of maximal curvature.  The relationship becomes more
than a visual similarity when the Dubins problem is formulated as a
left-invariant control problem on the group $SE(2)$ of orientation
preserving isometries of the plane or when extended to hyperbolic
space \cite{mittenhuber1998dubins}, \cite{monroy1998non}.

\tikzfig{dubins}{
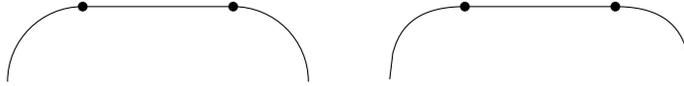
This research was motivated by the visual similarity between solutions to
the Dubins car problem with its circular arcs (left) and smoothed polygons 
using hyperbolic arcs (right).
}{
\begin{scope}
\draw (-1,0) -- (1,0);
\smalldot{-1,0};
\smalldot{1,0};
\draw (2,-1) arc (0:90:1);
\draw (-1,0) arc (90:180:1);
\end{scope}
\begin{scope}[xshift=2in]
\draw (-1,0) -- (1,0);
\smalldot{-1,0};
\smalldot{1,0};
\draw plot[domain=0:1] ({\x+1},{-1+(-5 + 5*\x + 3*sqrt(9-10*\x + \x*\x))/4});
\draw plot[domain=0:1] ({-\x-1},{-1+(-5 + 5*\x + 3*sqrt(9-10*\x + \x*\x))/4});
\end{scope}
}

The main results of this article are Theorem \ref{thm:pmp}, which
asserts that the smoothed $6k+2$-gon is given by a Pontryagin extremal
trajectory;  Theorem \ref{thm:local-min}, which gives the strict
local optimality of the smoothed octagon; and Theorem \ref{thm:finite},
which proves that extremal trajectories that avoid the singular locus have
bang-bang controls with finitely many switches.  

The hyperbolic plane plays an important role in this article.
The connection with planar geometry comes through the group $\SL$,
which acts on the plane by affine transformations and on the hyperbolic
plane by isometries.

Many of the calculations are computer assisted, using Mathematica.
The computer code (about 1000 lines of source) has been posted to our
github repository (github.com/flyspeck).  Many explicit formulas that
are too long to print here can be found in the computer code that
accompanies this article.

It is with some regret that I publish this article prematurely before
completing a full solution to the Reinhardt conjecture.  I have not
encountered any obstacles to major further advances along these lines,
and I believe that optimal control theory should eventually lead
to a solution to the Reinhardt conjecture.  The final section proposes
possible end-games for this problem.

\subsection{review of earlier results}\label{sec:review}

We briefly review some of the main conclusions of
\cite{Reinhardt:1934} and \cite{hales2011reinhardt}.  A convex body in
Euclidean space is a compact convex set with nonempty interior.  A
{\it centrally symmetric} convex body $D$ is a convex body such that
$-D$ is a translate of $D$.  In this article, a disk means a convex
body in the plane and is not necessarily circular.  

Reinhardt proved the existence of a convex centrally symmetric disk
$\DR$ in the plane with the property that the density of its densest
lattice packing minimizes the density of the densest lattice packing
among all convex centrally symmetric disks in the plane.  The
Reinhardt problem is to determine the shape of $\DR$.  Reinhardt
conjectured that $\DR$ is a smoothed octagon.

The density is not changed by affine transformations of the plane.
Thus if $\DR$ is a solution to the Reinhardt problem, then every affine
transformation of $\DR$ is also a solution.

Reinhardt showed that $\DR$ has no corners; that is, every point on
the boundary of $\DR$ has a unique tangent.  He showed that the
densest lattice packing is obtained by placing $\DR$ in a centrally
symmetric hexagon of smallest area containing $\DR$, then tiling the
plane with copies of the hexagon.  Moreover, there is a centrally
symmetric hexagon of the same minimal area passing through each point
on the boundary of $\DR$.  These hexagons never degenerate to a
quadrilateral.  By rescaling, we may assume without loss of generality
that the centrally symmetric hexagons all have area $\sqrt{12}$ (the
area of the circumscribing hexagon of circle of radius $1$) and that
the centrally symmetric disks are centered at the origin.  This puts
a structure on $\DR$ that we call a {\it hexagonally symmetric}
disk.  (This was called a hexameral domain in
\cite{hales2011reinhardt}.)

Let $\ee{0},\ee{1},\ldots,\ee{5}\in\ring{R}^2$ be the vertices of a
regular hexagon on  a unit circle:
\[
\ee{j} = (\cos(2\pi j/6),\sin(2\pi j/6)).
\]
Let $\SL$ be the group of $2\times 2$ matrices with real coefficients
with determinant $1$, and let $\sl$ be its Lie algebra, consisting of
all $2\times 2$ matrices with real coefficients and trace $0$.  We
write $t_f$ for the free terminal time (to be determined as part of
the solution).  After centering $\DR$ at the origin, there exists a
continuously differentiable path $g:[0,t_f]\to \SL$ such that the
boundary of $\DR$ is given by the six arcs
\begin{equation}\label{eqn:sigma}
t\mapsto \sigma_j(t):= g(t) \ee{j}, \quad t \in [0,t_f].
\end{equation}
For each $t$, we may draw the tangents to the boundary of $\DR$ at the
six points $g(t) \ee{j}$, for $ j=0,\ldots,5$.  The six tangents form
the six edges of an area-minimizing centrally symmetric hexagon as
above.  The six points are the midpoints of the six edges of the
hexagon.

\begin{remark} We work with the group $\SL$, but because of the
  central symmetry, we have $\sigma_{j+3}=-\sigma_j$, and we might pass to the
  quotient $\op{PSL}_2(\ring{R})$ if desired.
\end{remark}

Each choice of path $g$ (subject to the convexity and endpoint
constraints of Section \ref{sec:upper-half}) leads to a centrally
symmetric disk that possesses such a family of area-minimizing
centrally symmetric hexagons.  We call a path $g$ in $\SL$ hexagonally
symmetric (see Section \ref{sec:upper-half}) if it satisfies the
convexity and endpoint conditions to define a convex centrally
symmetric disk $D(g)$ in the plane, with boundary arcs
(\ref{eqn:sigma}).  Each hexagonally symmetric disk $D$ (that is, a
centrally symmetric convex disk in the plane, centered at the origin)
has the form $D=D(g)$ for some $g:[0,t_f]\to\SL$.

With fixed area $\sqrt{12}$ for each hexagon, the Reinhardt problem
becomes equivalent to minimizing the area of a hexagonally-symmetric disk.  This
can be formulated as a problem in the calculus of variations, as was
done in an earlier article, but the convexity constraints on the disk
lead to a some awkwardness.  We turn to optimal control theory as a
natural framework for the Reinhardt problem.

\section{State}

\subsection{ODE}

We consider the following control problem.  
Let 
\[
U = \{(u_0,u_1,u_2)\in \ring{R}^3\mid \sum_i u_i = 1,\ u_i\ge 0\}
\]
be the set of controls, a $2$-simplex.
We define an affine map 
\[
Z_0: U\to \sl
\]
as the inverse of the linear map $\sl\to\ring{R}^3$, $Z\mapsto
\ee{2i}\wedge Z \ee{2i}$; that is given $u\in\ring{R}^3$, by solving
equations for $Z_0$:
\begin{equation}\label{eqn:u}
  \ee{2i} \wedge Z_0(u) \ee{2i} = u_i,\quad i = 0,1,2,
\end{equation}
where $v \wedge v'$ is the $2\times 2$ determinant with columns $v$
and $v'$.  (This system of linear equations for $Z_0$ is nonsingular.)
We refer to $Z_0(u)$ as the control matrix.

Let $g:[0,t_f]\mapsto \SL$ be a $C^1$ path 
We write
\begin{equation}\label{eqn:g}
g' = g X, \text{ with } X:[0,t_f]\to \sl,
\end{equation}
We use a prime throughout the article to indicate the derivative
with respect to $t$.
We assume that
\begin{equation}\label{eqn:unit}
\det(X(t))=1,\quad \text{ for all } t\in [0,t_f].
\end{equation}
We assume that $X:[0,t_f]\mapsto \sl$ is Lipschitz continuous and that
\begin{equation}\label{eqn:X'}
X' = X (\delta(u,X) Z_0(u) - X),
\end{equation}
where 
\begin{equation}\label{eqn:delta}
\delta = \delta(u,X) = -2/\op{trace}(Z_0 (u) X).
\end{equation}

As a rough guide our intuition, we can view the ODE
(\ref{eqn:u})--(\ref{eqn:delta}) as a Frenet-Serret type formula that
determines a planar curve up to congruence by its planar curvature.
In our setting, the control $u=(u_0,u_1,u_2)$ gives the planar
curvatures of the various branches $\sigma_{2i}$ of the hexagonally
symmetric curve up to a normalization factor that has been included to
make $U$ a standard simplex.  More precisely, the curvature is given
as
\begin{equation}\label{eqn:kappa}
\kappa_{2i} = \left(\frac{dt}{ds_{2i}}\right)^3 \delta(u,X) u_i,
\end{equation}
where $\kappa_{2i}$ is the curvature of the $2i$-th curve $t\mapsto
g(t) \ee{2i}$, and $s_{2i}$ is its arclength parameter.  The
non-negativity conditions $u_i\ge 0$ are the local convexity
conditions on the hexagonally symmetric curve.

\begin{remark}
  Under natural disk constraints on $X$ that can be made without
  loss of generality, we show in Section \ref{sec:star} that the
  denominator of Equation (\ref{eqn:delta}) is nonzero.
\end{remark}

\begin{theorem} Let $g:[0,t_f]\to \SL$ be related by (\ref{eqn:sigma}) to
the solution  of the Reinhardt problem: $\DR=D(g)$.
  After a suitable reparametrization, the path $g$ satisfies the the
  equations (\ref{eqn:u})--(\ref{eqn:delta}) for some measurable control
  $u:[0,t_f]\to U$.
\end{theorem}

\begin{proof}
  We briefly indicate why an optimal solution to the Reinhardt problem
  gives a trajectory of this state equation.  The path $g$ is
  continuously differentiable.  Define $X:[0,t_f]\to\sl$ by $g' = g
  X$.  By \cite[\S3.3]{hales2011reinhardt}, $X$ is Lipschitz
  continuous, so that $X$ is differentiable almost everywhere.  By
  \cite[\S3.5]{hales2011reinhardt}, $\det(X)>0$.  The cost (that is, the
 area of a disk $D$) described in
  Section
  \ref{sec:cost} is invariant under reparametrizations of the path
  $g$.  By appropriate choice of time parameter for the path $g$, we
  may assume that $g$ has unit speed in the sense that Equation
  (\ref{eqn:unit}) holds.  Define a ``curvature matrix''
  $Z:[0,t_f]\mapsto\op{gl}_2(\ring{R})$
\begin{equation}\label{eqn:Z}
Z = X + X^{-1}X'.
\end{equation}
By Lemma \ref{lemma:trace0}, $Z$ takes values in $\sl$.  By
\cite[Eqn.19]{hales2011reinhardt}, we have
\begin{equation}
\ee{2j} \wedge Z \ee{2j} = \delta u_j \ge 0,
\end{equation}
for some $\delta>0$ and some measurable $u:[0,t_f]\to U$.  Define
$Z_0=Z_0(u)$ by Equation (\ref{eqn:u}), so that $Z = \delta Z_0$.  Solving
Equation (\ref{eqn:Z}) for $X'$, we obtain the differential equation
(\ref{eqn:X'}).  The scalar $\delta$ is uniquely determined by the
condition that $\op{trace}(X')=0$:
\begin{align*}
0 &= \op{trace}(X') = 
\op{trace}(X \delta Z_0-X^2) = \delta\op{trace}(X Z_0) +2.
\end{align*}
\end{proof}

\begin{lemma}\label{lemma:trace0} 
  If $X:[0,t_f]\to \sl$ is a Lipschitz path such that Equation
  (\ref{eqn:unit}) holds, then
\[
\op{trace}(X + X^{-1} X') = 0.
\]
\end{lemma}

\begin{proof} 
The characteristic polynomial of $X$ is
\[
\lambda^2 - \op{trace}(X)\lambda + \det(X) = \lambda^2+1.
\]
By Cayley-Hamilton, $X^2=-I$ and $X^{-1}=-X$.  
Differentiation gives $X X' + X' X=0$.
Then
\[
0 = \op{trace}(-X X') = \op{trace}(X + X^{-1} X').
\]
\end{proof}

\subsection{Poincar\'e upper half-plane}\label{sec:upper-half}

If $X$ is any matrix, we write
 $c_{ij}(X)$ for the $ij$ matrix coefficient of $X$.
In particular, we have linear functions $c_{ij}:\sl\to\ring{R}$:
\[
X = \begin{pmatrix}c_{11}(X)& c_{12}(X) \\ 
c_{21}(X) & - c_{11}(X)\end{pmatrix}.
\]
We say that $X\in\sl$ is positively oriented if $c_{21}(X) >0$.
 By \cite[\S3.5]{hales2011reinhardt}, a solution $X(t)$ to the Reinhardt
problem has positive orientation for each $t$.  
(This corresponds to a counterclockwise
traversal of the boundary of $\DR$.)

Set
\[
J = \begin{pmatrix} 0 & -1 \\ 1 & 0\end{pmatrix} \in \sl.
\]
We have rotation matrices
\[
\exp(J t) = \begin{pmatrix} \cos t & -\sin t\\ \sin t & \cos t\end{pmatrix}.
\]

\begin{lemma} 
The set of matrices $X\in\sl$ such
  that $\det(X)=1$, $\op{trace}(X)=0$ and $c_{21}(X)>0$ is the adjoint orbit
  of $J$.
\end{lemma}

\begin{proof}
  Let $y = 1/c_{21}$ and $x = c_{11}/c_{21}$.  Then $X$ has the form
\begin{equation}\label{eqn:orbit}
X = X(x,y)=
 \begin{pmatrix} x/y & -x^2/y - y \\ 1/y & -x / y\end{pmatrix} 
= \hat z J \hat z^{-1},
\quad \text{where} \quad \hat z = \hat z(x,y) = 
\begin{pmatrix} y & x \\ 0 & 1\end{pmatrix}.
\end{equation}
The centralizer of $J$ in $\SL$ is $\SO$.  By the
Iwasawa decomposition, the orbit of $J$ under $\SL$ is the same as the
orbit under the upper triangular matrices $\hat z(x,y)$ with positive
determinant.  The result follows.  (We remark that the condition $c_{21}>0$
picks out a conjugacy class within the stable semisimple conjugacy class
determined by the characteristic polynomial $\lambda^2 + 1$ of $X$.)
\end{proof}

The adjoint orbit of $J$ can be identified with the homogeneous space
$\SL/\SO$, which can be identified with the upper
half-plane $\h$.  This identification comes via $\hat z= z(x,y)$ as above:
\begin{equation}\label{eqn:X-h}
\hat z J \hat z^{-1} \mapsto  x + i y\in\h,\quad y>0,\quad i = \sqrt{-1}.
\end{equation}

A calculation shows that the ODE in Equation (\ref{eqn:X'}) expressed
in terms of coordinates $x,y$ is
\begin{align}\label{eqn:ode-state-upper}
\begin{split}
x' &= f_1(x,y;u):=
\frac{y (b + 2 a x - c x^2 + c y^2)}{b + 2 a x - c x^2 - c y^2}\\
y' &= f_2(x,y;u):=
\frac{2 (a - c x) y^2}{b + 2 a x - c x^2 - c y^2}.
\end{split}
\end{align}
The dependence on the control $u=(u_0,u_1,u_2)\in U$ comes through the
coefficients $a,b,c$.  Specifically, $Z_0 = Z_0(u)=\begin{pmatrix} a &
  b \\ c & -a\end{pmatrix}$ is the control matrix where
\begin{align*}
a &= c_{11}(Z_0) = \frac{u_1-u_2}{\sqrt{3}},\\
b &= c_{12}(Z_0) = \frac{u_0}{3} - \frac{2u_1}{3} - \frac{2 u_2}{3},\\
c &= c_{21}(Z_0) = u_0.
\end{align*}

In summary,
we have state equations:
\begin{align}\label{eqn:state}
\begin{split}
x' &= f_1(x,y;u),\\
y' &= f_2(x,y;u),\\
g' & = g X,\\
x&:[0,t_f]\to\ring{R};\quad y:[0,t_f]\to\ring{R};\quad g:[0,t_f]\to\SL;
\end{split}
\end{align}
where  $X = \hat z(x,y) J \hat z(x,y)^{-1}$, and
where $u:[0,t_f]\to U$ is a measurable control.

We define an {\it admissible} trajectory $(g,z):[0,t_f]\to M:=
\SL\times\h$ to be a solution of this ODE for some measurable control
$u$, with $z=x+iy$, and that satisfies the following additional
conditions:
\begin{enumerate}
\item The image of $z$ lies in $\hstar\subset \h$ (see Section
  \ref{sec:star}).
  \item The endpoints of the trajectory are $(g(0),z(0))=(I,z_0)$ and
    $(g(t_f),z(t_f))=(R,R^{-1}.z_0)$, for some $z_0\in \h$, where $R:=
    \exp(J\pi/3)$.  
\item the path
  $g:[0,t_f]\to\SL$ is homotopic in $\SL$ to the path
  given by rotation:
\[
t\mapsto \exp(J \pi t/(3 t_f)),\quad t\in [0,t_f].
\]
\end{enumerate}
To each admissible trajectory we may associate a hexagonally symmetric
disk $D(g,z)$.  The second condition enforces that the union of the
boundary arcs (Equation \ref{eqn:sigma}) of $D(g,z)$ is a closed curve
with no corners.  The third condition enforces the condition that the
boundary arcs must define a simple closed curve traversed in the
counterclockwise direction.

The path $g$ determines $z$ by the equations (\ref{eqn:g}) and
(\ref{eqn:orbit}).  Conversely $z$ determines $g$ by the same
equations and the initial condition
\[
g(0)=I.
\]
Thus, we sometimes abbreviate the admissible trajectory $(g,z)$ to $g$
or $z$, and write the corresponding hexagonally symmetric disk
$D(g,z)$ as $D(g)$ or $D(z)$.

The condition $g(0)=I$ can be imposed without loss of generality.  
The group of affine transformations of the plane acts on the set of 
solutions to the Reinhardt problem.  We have reduced the affine
group of symmetries by fixing the center of $\DR$ at the origin, and
the fixing the area $\sqrt{12}$ of the hexagon tile.  This leaves the
group action of $\SL$ on the set of solutions to the Reinhardt
problem, which we rigidify with the initial condition $g(0)=I$.

We call a {\it link} the full segment (between switching times) of a
trajectory that has a constant control at a vertex of the
simplex $U$: 
\[
u\in\{e_1,e_2,e_3\}\subset U.
\]
where $e_1=(1,0,0)$, $e_2=(0,1,0)$, $e_3=(0,0,1)\in U$.
See Sections \ref{sec:e2e3} and \ref{sec:e1}.

\subsection{cost functional}\label{sec:cost}

The cost functional in \cite[\S5.1]{hales2011reinhardt} (correcting
the formula there with a missing factor of $2$) is
\begin{equation}\label{eqn:cost0}
g\mapsto \op{area}(D(g))= -\frac{3}{2}\int_0^{t_f} \op{trace} (J g^{-1} g') dt \to\min.
\end{equation}
The interpretation of the cost is the area of the hexagonally
symmetric disk $D(g)$.  Using $g' = g X$, the cost (\ref{eqn:cost0})
depends only on $X$ and simplifies to
\begin{equation}\label{eqn:cost-JX}
 -\frac{3}{2}\int_0^{t_f} 
 \op{trace}(J X) dt  \to \min.
\end{equation}
Assume now that $X$ has unit speed.
Expressed in terms of coordinates $x + i y$ in the Poincar\'e upper
half-plane, the cost takes the form
\begin{equation}\label{eqn:cost-upper}
\frac{3}{2} \int_0^{t_f} \frac{x^2 + y^2 + 1}{y} dt\to \min.
\end{equation}
This cost is rotationally symmetric with respect to the action of
$\SO\le \SL$ on the upper half-plane.  In fact, the
level sets of $(x^2 + y^2 + 1)/y$ are concentric circles (centered at
$i$ in hyperbolic geometry).  The cost satisfies
\begin{equation}\label{eqn:cost-min}
(x^2 + y^2 + 1)/y \ge y + 1/y \ge 2,
\end{equation}
attaining its minimum at $i = \sqrt{-1}\in \h$.

We may also express the cost in the Poincar\'e disk model $\D$.  Let
\[
\D=\{w\in\ring{C} \mid |w| < 1\},\quad 
w = \frac{z - i}{z + i},\quad z = \frac{-i(1+w)}{-1+w},\quad z\in\h
\]
The cost of a
path $w:[0,t_f]\to \D$ in the Poincar\'e disk becomes
\begin{equation}\label{eqn:cost-disk}
3\int_0^{t_f} \frac{1 + |w|^2}{1 - |w|^2} dt\to\min.
\end{equation}
In this model, the rotational symmetry about $0$ is evident.

\begin{remark}
  One model of hyperbolic geometry is the upper sheet of a hyperboloid
  of two sheets.  We recognize $(1+|w|^2)/(1 - |w|^2)$ as the height
  on the hyperboloid \[ \{(u,v,h)\in \ring{R}^3\mid h^2 = 1 + u^2 +
  v^2,\ h>0\}.
 \]
 In more detail, we map a point $w=(u,v,0)\in\ring{R}^3$ in the unit
 disk $\D\subset \ring{R}^3$ to the point $p$ in the upper sheet whenever $w$, $p$, and
 $(0,0,-1)$ are collinear (Figure \ref{fig:hyperboloid}).  It is a
 curiosity that the area of a convex disk in the Euclidean plane in
 Reinhardt's packing problem equals the integral of the height
 function in hyperbolic geometry.
\end{remark}

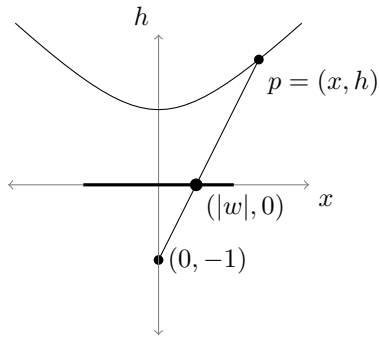
\tikzfig{hyperboloid}{ The cost is the integral of the height $h$ in
  the hyperboloid model of hyperbolic geometry.  The line through
  $(0,-1)$ and $(|w|,0)$ meets the hyperbola $h^2 = 1 + x^2$ at a
  point $(x,h)$ with height $h= (1+|w|^2)/(1-|w|^2)$.  }{
\begin{scope}
\draw (0,-1) -- (0.5,0) -- (4/3,5/3);
\smalldot{0,-1};
\smalldot{4/3,5/3};
\draw[fill=black] (0.5,0)
 node [inner sep=1.7pt,shape=circle,fill=black] {};
\draw (0.5,0) node [anchor=north west] {$(|w|,0)$};
\draw (0,-1) node [anchor=west] {$(0,-1)$};
\draw (4/3,5/3) node [anchor=north west] {$p=(x,h)$};
\draw[gray,<->] (-2,0)--  (2,0) node[anchor=north west,black] {$x$};
\draw[gray,<->] (0,-2)--(0,2) node[anchor=south east,black] {$h$};
\draw[very thick] (-1,0)--(1,0);
\draw plot[domain=-1.4:1.4] ({(exp(\x)-exp(-\x))/2},{(exp(\x)+exp(-\x))/2.0});
\end{scope}
}

\subsection{star inequalities}\label{sec:star}

We define the star inequalities on $\sl$ to be the following:
\[
\sqrt{3}|c_{11}(X)| < c_{21}(X),\quad 3 c_{12}(X) + c_{21}(X) < 0.
\]
(These conditions were obtained in \cite[\S3.5]{hales2011reinhardt}
for the tangent $X$ at $t=0$, but also hold for all $t$, by symmetry.)
There is no loss of generality in imposing these conditions at each
point of a trajectory; they are necessary conditions for the convexity
of the corresponding hexagonally symmetric disk (Remark
\ref{rem:star}).  Translating into the upper-half plane coordinates,
the star inequalities define an open region:
\begin{equation}\label{eqn:star}
\hstar:= \{(x,y)\in\h\mid 
-\frac{1}{\sqrt{3}} < x < \frac{1}{\sqrt{3}},\quad \frac{1}{3} < x^2 + y^2\}.
\end{equation}
The inequalities define the interior of an ideal hyperbolic triangle
with vertices at $z = \pm 1/\sqrt{3}$ and $z=\infty$ on the boundary
of $\h$.  We set $\Mstar = \SL\times\hstar$.  In disk coordinates, the
star inequalities imply that $w$ lies in the interior of the ideal
hyperbolic triangle with vertices $w=1,\zeta,\zeta^2$ on the boundary
of $\D$, where $\zeta = e^{2\pi i/3}$ (Figure \ref{fig:ideal}).

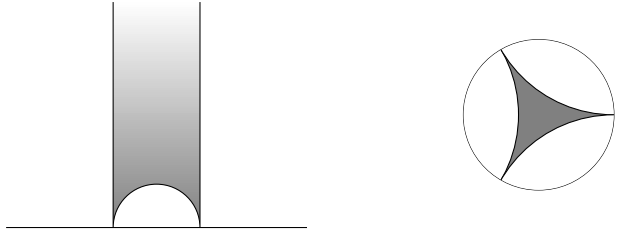
\tikzfig{ideal}{The star inequalities define an ideal triangle in
  gray, shown here in the upper half-plane and disk models of
  hyperbolic geometry.}  {
\def\rt{1.732}
\begin{scope}[xshift=0in,yshift=0in]
\shade[top color=white,bottom color=gray] (-1/\rt,0) rectangle (1/\rt,3cm);
\draw (-1/\rt,0)--(-1/\rt,3);
\draw (1/\rt,0)--(1/\rt,3);
\begin{scope}
\clip (-1/\rt,0) rectangle (1/\rt,3cm);
\draw[fill=white] (0,0) circle (0.577cm);
\end{scope}
\draw (-2,0) -- (2,0);
\end{scope}
\begin{scope}[xshift=2in,yshift=1.5cm]
\draw[fill=gray] (0,0) circle (1cm);
\clip (0,0) circle (1cm);
\draw[fill=white] (-2,0) circle (\rt cm);
\draw[fill=white] (1,\rt) circle (\rt cm);
\draw[fill=white] (1,-\rt) circle (\rt cm);
\end{scope}
}

\begin{remark}\label{rem:star}
  Every hexagonally symmetric disk that passes through the six points
  $\ee{i}$, for $i=0,\ldots,5$ must be contained in the six triangular
  petals of the hexagram (Figure \ref{fig:star}).  The star
  inequalities can be interpreted geometrically as asserting that the
  tangent vectors at $\ee{i}$ are nonzero and point into the open
  cones over the six triangular petals.  (When some tangent vector at
  some $\ee{i}$ points into the boundary of the open cone, the centrally
  symmetric hexagon $H$ degenerates to a parallelogram, which is never
  optimal.)
\end{remark}

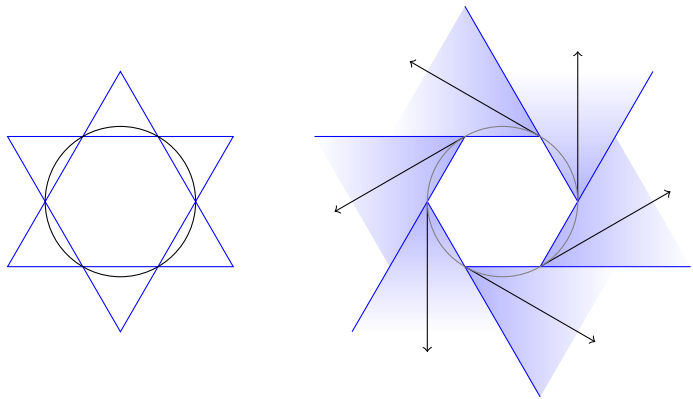
\tikzfig{star}{After an affine transformation, every hexagonally
  symmetric disk remains confined to the six petals of a hexagram on
  the left. The six tangents to the disk are confined to the open
  cones over the six petals. These are the star inequalities.}
{
\def\rt{1.732}
\def\R{2}
\begin{scope}
\draw[blue] (0,-\rt) -- (1.5,\rt/2) -- (-1.5,\rt/2) -- cycle;
\draw[blue] (0,\rt)-- (-1.5,-\rt/2) -- (1.5,-\rt/2) --cycle;
\draw (0,0) circle (1cm);
\end{scope}
\begin{scope}[xshift=2in]
\shade[left color=blue!30,right color=white] (0.5,-\rt/2) -- ++(0:2) -- ++ (120:2) --cycle;
\shade[bottom color=blue!30,top color = white] (1,0) -- ++ (60:2) -- ++(180:2) -- cycle;
\shade[right color=blue!30,left color = white] (0.5,\rt/2) -- ++ (180:2) -- ++(60:2) -- cycle;
\shade[right color=blue!30,left color = white] (-0.5,\rt/2) -- ++ (180:2) -- ++(-60:2) -- cycle;
\shade[top color=blue!30,bottom color = white] (-1,0) -- ++ (-120:2) -- ++(0:2) -- cycle;
\shade[left color=blue!30,right color = white] (-0.5,-\rt/2) -- ++ (0:2) -- ++(-120:2) -- cycle;
\draw[blue] (-60:1) -- ++ (1*60:3);
\draw[blue] (0:1) -- ++ (2*60:3);
\draw[blue] (60:1) -- ++ (3*60:3);
\draw[blue] (2*60:1) -- ++ (4*60:3);
\draw[blue] (3*60:1) -- ++ (5*60:3);
\draw[blue] (4*60:1) -- ++ (6*60:3);
\draw[->] (1,0)-- +(90:\R);
\draw[->] (0.5,\rt/2) -- +(150:\R);
\draw[->] (-0.5,\rt/2) -- +(210:\R);
\draw[->] (-1,0) -- +(270:\R);
\draw[->] (-0.5,-\rt/2) -- +(-30:\R);
\draw[->] (+0.5,-\rt/2) -- +(30:\R);
\draw[gray] (0,0) circle (1cm);
\end{scope}
}

\begin{lemma}\label{lemma:denneg} 
  Let $Z_0(u)\in\sl$ be the control matrix of $u\in U$.  If $X\in \sl$
  satisfies the star inequalities, then $\op{trace}(Z_0(u) X)<0$.
\end{lemma}

\begin{proof} The control simplex $U$ is convex, and $Z_0(U)\subset\sl$
  is an affine image of the control simplex.  Thus, the image $Z_0(U)$
  is a convex set. It is enough to check that $\op{trace}(Z_0(u) X)<0$
  at the three vertices of the control simplex.  The conditions at
  vertices are precisely the star inequalities on $X$.
\end{proof}

\section{Costate}

\subsection{Filippov's lemma}\label{sec:fil}

The existence of lifted trajectories in the cotangent bundle (as
discussed later in this section) is generally based on Filippov's
compactness lemma \cite[Th.10.1]{agrachev2013control},
\cite[\S4.5]{liberzon2012calculus}.  In this subsection, we show that
the assumptions of Filippov's lemma are fulfilled.  Filippov's lemma
requires (1) that the control set $U$ is compact, which is certainly
true in our situation.

Filippov's lemma requires (2) that for each $x+i y\in \hstar$ 
 the velocity set (see Equation
\ref{eqn:ode-state-upper})
\[
\{f(x,y;u)=(f_1(x,y;u),f_2(x,y;u))\in \ring{R}^2\mid u\in U\}
\]
is convex.  We prove that the velocity set is in fact the
convex hull of 
\[
\{ f(x,y;e_k)\mid k=1,2,3\}.
\]
Fix $x+i y\in \hstar$ and pick two vertices $e_i,e_j\in U$.  
By explicit calculation, the two vertices map to distinct points
in the velocity set.
Let $L:\ring{R}^2\to \ring{R}$ be the nonzero affine function that
vanishes at $f(x,y;e_i)$ and $f(x,y;e_j)$.  From the explicit form of
Equation \ref{eqn:ode-state-upper}, we have
\[
L(f(x,y;u)) = \frac{\ell_1( u)}{\ell_2(u)},
\]
for some affine functions $\ell_1,\ell_2:U\to \ring{R}$ (depending on
$x,y$), where $\ell_2(u)$ is nonvanishing (and fixed sign) on $U$.  By
direct calculation,
we obtain $\ell_1(u)=0$ along the segment $[e_i,e_j]\subset U$ and
that $\ell_1$ has fixed sign on $U$.  We conclude that the velocity
set is a  convex hull as claimed.

Finally, Filippov's lemma requires (3) the compact support in
$(x,y)$ of the velocity sets.  This is a serious issue in our setting
because the star inequalities are open conditions and the vector
fields become unbounded near the boundary.  

By Reinhardt, an optimal centrally symmetric $\DR$ exists and its
boundary has no corner.  The corresponding unit-speed trajectory
$z:[0,t_f]\to\h$ remains in the interior of some compact set $K\subset
\hstar$.  A standard argument using a smooth compactly-supported
support function $f:\hstar\to \ring{R}$ with $f|_K = 1$ allows us to
replace each vector field $F$ on $\h$ with a vector field $f F$ of
compact support \cite[Remark 10.5]{agrachev2013control}.  Thus, by
choosing a suitable support function, we may assume that all three of
Filippov's assumptions hold.  Moreover, if desired, we can exhaust
$\hstar$ by a sequence of compact sets $K$ whose union is $\hstar$.

Pontryagin's conditions, which are discussed below, are
local around the trajectory $z$, and so are not affected by the
support function.

\subsection{Hamiltonian}\label{sec:ham}

We use the formulation of the Hamiltonian for invariant problems on a
Lie group from \cite[Ch.18]{agrachev2013control}.  We
use invariant vector fields to trivialize the tangent bundle of $\SL$
and use an invariant inner product to identify the cotangent space
with the tangent space.  We fix the invariant inner product $\ang{A,B}
= \op{trace}(AB)$ on $\sl$.  In this formulation, according to the standard
definitions,  the optimal control
problem has a Hamiltonian
\begin{align}\label{eqn:ham}
\begin{split}
H(\lambda;u)&:=
\Hlie(\Lambda,X) + H_{\h}(\nu,x,y;u), 
~\text{where}\\
\Hlie(\Lambda,X) &:= 
\ang{\Lambda,X} - \frac{3}{2} \lambda_{cost} \ang{J,X},\\
H_{\h}(\nu,x,y,u)&:= 
\nu_1 f_1(x,y;u) + \nu_2 f_2(x,y;u),\\
\end{split}
\end{align}
 for
costate variables $\lambda_{cost}\in\ring{R}$, $\Lambda\in \sl$, 
$\nu=(\nu_1,\nu_2)\in\ring{R}^2$,
and where $X = \hat z J \hat z^{-1}$, $\hat z = \hat z(x,y)$.  We
have broken the Hamiltonian into two terms: $H_{\h}$ coming from the
upper-half plane, and $\Hlie$ coming from the Lie algebra and cost
functional combined.

The state space $\Mstar:=\SL\times \hstar$ is five-dimensional, and
  the costate space $\sl\times\ring{R}^2$ is five-dimensional,
viewed as the
cotangent space of $\Mstar$ at a point under the trivialization of the
cotangent bundle.  We write $T^*M$ for the cotangent bundle of $M$,
identified with
\[
T^*M\times T^*\h = (\SL\times\sl)\times (\h\times\ring{R}^2)
\ni((g,\Lambda),(z,\nu)).
\]
We write $T_I^*M$ for the subspace of $T^*M$ on which the $\SL$
component is $g=I$.

\subsection{Pontryagin maximum principle}\label{sec:pmp}

Specialized to our setting,
the conditions of Pontryagin maximum principle (PMP) for
our optimal control problem with free terminal time $t_f$ are
the following:
\begin{enumerate}
\item The trajectory $(g,z)$ satisfies the ODE
  (\ref{eqn:state}) for some measurable control $u:[0,t_f]\to U$.
\item The  Hamiltonian $H(\lambda,u)$ vanishes identically along the
  lifted controlled trajectory $(\lambda,u)$.  
\item The lifted trajectory $\lambda:[0,t_f]\to T^*M$ is Lipschitz continuous
  and satisfies the following ODE:
\begin{equation}\label{eqn:adjoint}
\begin{split}
\Lambda' &= [\Lambda,X],\\
\nu_1' &= -\partials{H^+}{x},\\
\nu_2' &= -\partials {H^+}{y},
\end{split}
\end{equation}
Here $H^+$ is the
  pointwise maximum over the control simplex $U$:
\begin{equation}\label{eqn:pmp-max}
H^+(\lambda_{cost},\lambda)=
\max_{u\in U} H(\lambda_{cost},\lambda;u),\quad 
(\lambda_{cost},\lambda)\in\ring{R}\times T^*\Mstar.
\end{equation}
\item The projectivized covector is well-defined: for each $t$, the
  vector $(\lambda_{cost},\lambda(t))\in\ring{R}\times T^*M$ is nonzero.
\item $\lambda_{cost}$ is constant and $\lambda_{cost}\le 0$.
\item Transversality holds at the endpoints (as described in
  Section \ref{sec:trans}).
\end{enumerate}

By a lifted trajectory $\lambda:[0,t_f]\to T^*M$ of an admissible
trajectory $(g,z)$ we mean a solution of the ODE (\ref{eqn:adjoint})
such that the image $(g_\lambda,z_\lambda)$ of $\lambda$ in $M$ is
$(g,z)$.

A lifted trajectory satisfying the PMP conditions is called a
Pontryagin extremal trajectory.  The PMP gives necessary but not
sufficient conditions for local optimality.

Because $\lambda_{cost}\le 0$ is a constant, and since the PMP
conditions are invariant under rescaling the costate by a positive
scalar, we may take $\lambda_{cost}=0$ (abnormal multiplier) or
$\lambda_{cost}=-1$ (normal multiplier).

We define a {\it Reinhardt trajectory} to be a trajectory $(g,z)$ such
that its disk $D(g,z)=\DR$ is a globally optimal solution to the
Reinhardt problem.

\begin{lemma} Let $(g,z)$ be a Reinhardt trajectory.
Then the trajectory has a lifting
\[
\lambda:[0,t_f]\to T^*\Mstar, \quad\lambda_{cost}\in \{0,-1\}
\] to the cotangent space.  The lifted trajectory is a Pontryagin
extremal trajectory.
\end{lemma}

\begin{proof} 
  Filippov's lemma gives a Lipschitz continuous path
  $\lambda:[0,t_f]\to T^*M$, with components
  $\lambda=(\Lambda,\nu)$ satisfying the
  adjoint equations (\ref{eqn:adjoint}):
(The first equation in (\ref{eqn:adjoint}), for $\Lambda'$, comes from
a trivialization of the cotangent bundle of $\SL$, as stated in
\cite[Eqn.18.18]{agrachev2013control}.)

By the work of Pontryagin and general control theory, 
the PMP are necessary conditions
for optimality.
\end{proof}

\subsection{rotational symmetry}\label{sec:sym}

We write $(\rho,z)\mapsto \rho.z$ for the action of $\SL$ on $\h$ by
linear fractional transformations.

Let $\rho\in \SO$ be any rotation.  The symmetry acts on trajectories
and related data as follows.  Let $\lambda=(\Lambda,\nu,\ldots)$ be a
lifted trajectory.  We map the path $z=x+iy$ in the upper-half plane to
the path $\bar z = \rho.z$, where we use bars to denote transformed
quantities.  Then short calculations show that we obtain another
lifted trajectory $\bar\lambda=(\bar \Lambda,\bar \nu,\ldots)$ (with
different boundary values) and associated parameters:
\begin{align*}
\bar g &= \rho g \rho^{-1};\\
\bar X &= \rho X \rho^{-1};\\
\bar z &= \rho.z;\\ 
\bar \Lambda &= \rho\Lambda \rho^{-1};\\
\bar Z_0 &= \rho Z_0 \rho^{-1};\\
\bar \delta &= \delta;\\
\end{align*}
The cost is invariant: $ {\op{\={c}ost}} = \op{cost}$.    
The transformation rule for 
$\nu$ is as follows.
The value of $\bar\nu$ at $\rho.z$ is 
\[
\bar\nu_{\rho.z} = (F^t)^{-1} \nu_z,\quad \nu_z\in T_z^*\h,
\]
where the linear map $F= d\rho_z$ of tangent spaces and transpose $F^t$
\[
F:T_z\h\to T_{\rho.z}\h,\quad F^t:T_{\rho.z}^*\h\to T_z^*\h.
\]
are induced from $z\mapsto \rho.z$.

It is remarkable that the entire Hamiltonian is invariant under the
full rotation group $\SO$:
\[
\bar H = H;\quad 
\ang{\bar\Lambda,\bar X} = \ang{\Lambda,X};\quad
(\bar \nu,\bar f)=(\nu,f);\quad
\op{trace}(J\bar X) = \op{trace}(J X).
\]

Moreover, assume that $\rho\in\langle R\rangle$. Then there exists a
permutation $\pi=\pi_\rho$ of $\{0,1,2\}$ such that
\[
\rho^{-1} \ee{2i} = \ee{2\pi_i}, \quad \text{for } i=0,1,2.
\]
Then
\begin{equation}\label{eqn:ubar}
\bar u =  (\bar u_0,\bar u_1,\bar u_2) = (u_{\pi 0},u_{\pi 1},u_{\pi 2}).
\end{equation}
We write $\bar u = \rho\cdot u$ for this action.

\subsection{terminal conditions}\label{sec:trans}

We define periodic boundary conditions (modulo rotation by $R$):
\[
g(0) = I,\quad g(t_f) = R := \exp(J\pi/3).
\]
and
\begin{equation}\label{eqn:periodic}
X(0) \in \sl,\quad X(t_f) = R^{-1} X(0) R \in \sl.
\end{equation}
The terminal condition $g(t_f)= R$ is necessary because the six paths
of the hexagonally symmetric disk must join together to give a closed
curve:
\begin{equation}\label{eqn:g-term}
g(t_f) \ee{j} = g(0) \ee{j+1}\quad \Leftrightarrow \quad g(t_f) = R.
\end{equation}
By the unit speed positive orientation conditions, 
\[
g(t+t_f) \ee{j} =
g(t) \ee{j+1} = g(t) R \ee{j},\quad
g(t + t_f) =g(t) R,
\] 
and taking derivatives: 
\begin{equation}\label{eqn:rotsym}
X(t+t_f) = R^{-1} X(t) R,\quad 
z(t+t_f) = R^{-1}.z(t).
\end{equation}
Evaluating (\ref{eqn:rotsym}) at
$t=0$ gives the terminal condition on $X(t_f)$ in Equation
\ref{eqn:periodic}.

Expressed in terms of coordinates on the upper half-plane, the
terminal condition becomes
\begin{equation}\label{eqn:z-term}
z(t_f) = R^{-1}. z(0),\quad z(0),z(t_f)\in \h,
\end{equation}
a rotation about $i$ by angle $2\pi/3$.
Expressed in terms of a complex variable in the Poincar\'e disk model,
the terminal condition becomes a counterclockwise rotation by angle
$2\pi/3$:
\[
w(t_f) = \zeta  w(0),\quad w(0),w(t_f) \in \D,\quad \zeta = e^{2\pi i/3}.
\]

In optimal control problems such as this with  free terminal time $t_f$,
Pontryagin's maximum principle (PMP) includes a transversality
condition at time $t=t_f$.  For periodic systems such as ours, the
transversality condition can be found in Liberzon
\cite[p134]{liberzon2012calculus}.  In our setting, the system is
periodic up to rotation by $R$.  Our transversality conditions can be
expressed as follows:
\begin{align}\label{eqn:transverse}
\begin{split}
\Lambda(t_f) &= R^{-1} \Lambda(0) R,\\
F^t\lambda(t_f) &= \lambda(0),\\
\end{split}
\end{align}
where $F=dR^{-1}_{z(0)}:T_{z(0)}\h\to T_{z(t_f)}\h$ is the 
linear map of tangent spaces
induced from $z\mapsto R^{-1}.z$ and 
 its transpose is 
\[
F^t:T^*_{z(t_f)}\h\to T^*_{z(0)}\h.
\]

\section{Explicit trajectories with bang-bang controls}

By a bang-bang control we mean a measurable control function $u$ that
takes values in the set $\{e_1,e_2,e_3\}$ of vertices of the control
simplex $U$.

\subsection{constant control at the vertex $e_2$ or $e_3$}
\label{sec:e2e3}

Throughout this section, we assume that the control $u\in U$ is
constant, fixed at a vertex $u=e_3$ or $u=e_2$ of $U$.
Under this assumption, we give
the general explicit formula for the state and costate.  
With such controls, the control matrix $Z_0$ in Equation
(\ref{eqn:u}) simplifies to the form
\[
c_{11}(Z_0) = \pm 1/\sqrt{3},\quad c_{12}(Z_0) = -2/3,
\quad c_{21}(Z_0) = 0,\quad m = c_{12}/(2 c_{11}) = \pm 1/\sqrt{3}.
\]
Specifically, $m=m_3=1/\sqrt{3}$ (control $u=e_3$) and $m_2= - 1/\sqrt{3}$,
for these two controls.

In this context, the ODE (\ref{eqn:ode-state-upper}) reduces to
\begin{equation}\label{eqn:ode-state-bang}
x'  = y;\quad y' = \frac{y^2}{m + x}.
\end{equation}
The star inequalities imply that $c_0:= x(0)+m\ne 0$. Set $\alpha =
y(0)/c_0$, which is also nonzero by the star inequalities.  The
general solution to ODE (\ref{eqn:ode-state-bang}) is
\begin{equation}\label{eqn:alpha}
x(t) = -m + c_0 e^{\alpha t},\quad y(t) = c_0 \alpha e^{\alpha t}.
\end{equation}
(We also write this curve as $t\mapsto z(z_0,t)\in \h$, where
$z_0=x(0)+iy(0)$.)  In particular, each trajectory traces out a line
$y = \alpha (x+m)$ through the fixed point $(-m,0)$.  See Figure
\ref{fig:traj-upper}.  The motion is away from the fixed point when
$m=m_3$ and towards the fixed point when $m=m_2$.  That is,
$\alpha_3>0$ and $\alpha_2<0$.

Set $s= e^{\alpha t}$.  Expressed in terms of the independent variable $s$, 
the differential equation (\ref{eqn:g}) takes the form
\[
\alpha s\frac{dg}{ds} = g X,\quad g(1)=I,
\]
which has the explicit solution
\[
g(s) = I + \frac{s-1}{\alpha^2 c_0 s}\begin{pmatrix}
c_0-m & * \\
1 & 
 m-c_0 s
\end{pmatrix},
\]
where the missing entry $(*)$ is determined by the condition
$\det(g(s))=1$.

The adjoint equation also has an explicit exact general solution,
which appears in the accompanying computer algebra calculations.
Although it is entirely explicit, the solution is a bit too long to
print here.  The function $\nu$ is a pair of polynomials in $t$,
$e^{\alpha t}$ and $e^{-\alpha t}$, and there are five constants of
integration (beyond $z_0$).  These constants are determined by the
initial vector
\[
\lambda(0)\in T_{(I,z_0)}^* M.
\]

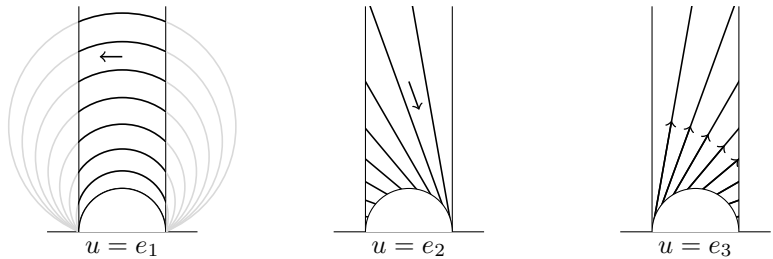
\tikzfig{traj-upper}{The trajectories with constant control $u=e_i$
  are circles or lines, shown here in the star region of the upper
  half-plane.}{
\def\rt{1.732}
\begin{scope}[xshift=3in,yshift=0in]
\draw (-1/\rt,0)--(-1/\rt,3);
\draw (1/\rt,0)--(1/\rt,3);
\draw (-1,0) -- (1,0);
\draw (0,0) node [anchor=north] {$u=e_3$};
\begin{scope}
\clip (-1/\rt,0) rectangle (1/\rt,3cm);
\foreach \t in {10,20,30,40,50,60,70,80} \draw[->,semithick]
(-1/\rt,0) -- ++ (\t:4);
\foreach \t in {10,20,30,40,50,60,70,80} \draw[->,semithick]
(-1/\rt,0) -- ++ (\t:1.5);
\draw[fill=white] (0,0) circle (0.577cm);
\end{scope}
\end{scope}
\begin{scope}[xshift=1.5in,yshift=0in]
\draw (-1/\rt,0)--(-1/\rt,3);
\draw (1/\rt,0)--(1/\rt,3);
\draw (-1,0) -- (1,0);
\draw (0,0) node [anchor=north] {$u=e_2$};
\draw[->,semithick] (0,2) -- ++ (-70:0.4);
\begin{scope}
\clip (-1/\rt,0) rectangle (1/\rt,3cm);
\foreach \t in {10,20,30,40,50,60,70,80} \draw[semithick]
(1/\rt,0) -- ++ (180-\t:4);
\draw[fill=white] (0,0) circle (0.577cm);
\end{scope}
\end{scope}
\begin{scope}[xshift=0in,yshift=0in]
\draw (-1/\rt,0)--(-1/\rt,3);
\draw (1/\rt,0)--(1/\rt,3);
\draw (-1,0) -- (1,0);
\draw (0,0) node [anchor=north] {$u=e_1$};
\clip (-3,0) rectangle (3,3);
\foreach \t/\r in {0/0.577,0.2/0.611,0.4/0.702,0.6/0.832,
0.8/0.986,1/1.154,1.2/1.33,1.4/1.51}
\draw[semithick,gray!30]
(0,\t) circle (\r);
\begin{scope}
\clip (-1/\rt,0) rectangle (1/\rt,3cm);
\foreach \t/\r in {0/0.577,0.2/0.611,0.4/0.702,0.6/0.832,
0.8/0.986,1/1.154,1.2/1.33,1.4/1.51,1.6/1.7}
\draw[semithick]
(0,\t) circle (\r);
\draw[->,semithick] (0,2.33) -- ++ (-180:0.3);
%
\draw[fill=white] (0,0) circle (0.577cm);
\end{scope}
\end{scope}
}

\subsection{constant control $u=e_1$}
\label{sec:e1}

The three extremal controls $u=e_1$, $e_2$, and $e_3$ are related by
rotational symmetry of the upper-half plane.  These symmetries are
more visually evident in the disk model of hyperbolic space, but the
solutions to the ODE take a simpler form in the upper-half plane.

Equation (\ref{eqn:ubar}) implies that we obtain the general
explicit solutions to the state and costate equations for control
$u=e_1$ by rotating solutions with control $e_2$ or $e_3$, as
described in Section \ref{sec:sym}.  
Details are found in the
computer code.

The trajectories with constant control $e_3$ move along Euclidean
lines through $(-1/\sqrt{3},0)$, which we view as circles through
$(-1/\sqrt{3},0)$ and $\infty$.  Under linear fractional
transformations, circles map to circles.  From this, we conclude that
trajectories with constant control $e_1$ must move along Euclidean
circles through the two fixed points $(\pm 1/\sqrt{3},0)$.

\subsection{bang-bang controls}\label{sec:bang}

\begin{lemma}\label{lemma:u-lambda} 
For every $\lambda\in T^*M$, the set of maximizers of
  the Hamiltonian:
\[
U_\lambda := \{u \in U \mid H(\lambda,u) = 
\max_{u\in U} H(\lambda,u) = H^+(\lambda)\}
\]
is a face of the convex set $U$; that is, $U_\lambda$ is a vertex, an
edge, or all of $U$.
\end{lemma}

\begin{proof}
  Fix $\lambda\in T^*M$.  The only term of the Hamiltonian that
  depends on the control is $H_{\h}$.  This term is the ratio of two
  linear functions on $U$.  For any $u_1,u_2\in U$, let $u(s) = s u_1
  + (1-s) u_2$ for $s\in [0,1]$ be a segment in the convex control
  set.  Then the dependence of the Hamiltonian along the segment has
  the form of a linear fractional transformation
\[
s\mapsto H(\lambda,u(s))=\frac{a s + b}{c s + d},
\]
with derivative
\[
\frac{ad - bc}{(c s + d)^2}
\]
of fixed sign.  (The denominator is nonzero by Lemma
\ref{lemma:denneg}.)  Thus the Hamiltonian is monotonic along every
segment in the control simplex $U$.  The Hamiltonian therefore assumes
its maximum along a face.
\end{proof}

If $X\in \sl$ lies in the orbit of $J$, let $X_\h\in \h$ be the
corresponding element of the upper-half plane under the bijection
(\ref{eqn:X-h}).  For $t\ge 0$ and $z\in\h$, let $\gamma_0(z,t)\in
\SL$ be the trajectory with constant control $u=e_3$ and initial
conditions
\[
\gamma_0(z,0)=I,\quad \gamma_0'(z,0)_\h =  z.
\]
 (As always, prime denotes the
$t$ derivative.)
Let $\gamma_i(z,t)\in \SL$, for $t\ge 0$, $i\in\ring{Z}$, and $z\in
\h$ be the trajectory
\[
\gamma_i (z,t) := R^i \gamma_0(z,t) R^{-i}.
\]
We have 
\[
\gamma_i(z,0) = I,\quad 
(\gamma_i'(z,0))_\h = R^i.z,
\]
with constant control $u=R^i\cdot e_3$, using the action
(\ref{eqn:ubar}) of the cyclic group $\langle R\rangle$ on the control
simplex $U$. 

We define a continuous (shifted) extension of $\gamma_i$ that is
non-constant only for $t\in [T_1,T_2]$:
\[
\gamma_i(z,T_1,T_2,t) := \begin{cases} 
I, & \text{if } t \le T_1;\\
\gamma_i(z,t-T_1), &\text{if } T_1\le t\le T_2;\\
\gamma_i(z,T_2-T_1), &\text{if } T_2\le t.
\end{cases}
\]
The derivative $\gamma_2'$ has jump discontinuities at $T_1$ and
$T_2$.  Let $z(z_0,t)$ be the solution to the ODE
(\ref{eqn:ode-state-bang}) with constant control $u=e_3$ and
initial condition $z_0$.  For any tuple 
\[
\kappa =
((k_1,t_1),(k_2,t_2),\ldots,(k_{n},t_{n}))
\]
 with $k_i\in\ring{Z}$ and
$t_i\ge 0$, and for any $z_0\in \hstar$, let
\begin{align}
\begin{split}
T_0 &= 0;\\
T_{i+1} &= T_i + t_{i+1};\\
z_{i} &= R^{k_i -k_{i+1}}. z(z_{i-1},t_i);\\
\gamma(\kappa,z_0,t) &= \gamma_{k_1}(z_0,T_0,T_1,t) 
\gamma_{k_2}(z_1,T_1,T_2,t)\cdots \gamma_{k_n}(z_{n-1},T_{n-1},T_n,t).
\end{split}
\end{align}
Note that on the right-hand side of the last equation, only one factor
at a time is non-constant.  Then $\gamma(\kappa,z,t)$ is continuous in
$t$ and has unit speed parametrization.  Set $X(\kappa,z,t):=
\gamma(\kappa,z,t)^{-1} \gamma'(\kappa,z,t)$.  Note that for $t\in
[T_{i-1},T_{i}]$, when the $i$th factor is active, we have
\begin{align*}
X(\kappa,z_0,t) &= \gamma_{k_i}(z_{i-1},T_{i-1},T_{i},t)^{-1} 
\gamma_{k_i}'(z_{i-1},T_{i-1},T_{i},t) \\
  &=\gamma_{k_i}(z_{i-1},t-T_{i-1})^{-1}\gamma_{k_i}'(z_{i-1},t-T_{i-1})\\
  &= R^{k_i} X(z_{i-1},t - T_{i-1}) R^{-k_i},
\end{align*}
where $X(z,t) = \gamma_0(z,t)^{-1} \gamma_0'(z,t)$.  Comparing left
and right limits of $X(\kappa,z_0,t)$ at the boundary value $t=T_i$, we
find that $X(\kappa,z_0)$ is continuous in $t$:
\begin{align*}
X(\kappa,z_0,T_i^-)_\h &= R^{k_i}. z(z_{i-1},t_i) = R^{k_{i+1}}.z_i;\\
X(\kappa,z_0,T_i^+)_\h &= R^{k_{i+1}}.z_i.
\end{align*}

From this, it is easy to see that $\gamma(\kappa,z_0)$ is the general
bang-bang trajectory with finitely many switches (at times
$T_0,\ldots, T_n$), as we vary $\kappa$ and $z_0$. The control on the
interval $[T_{i-1},T_{i}]$ is $u=R^{k_i}\cdot e_3\in U$.

The total $\op{cost}(z_0,[0,t])$ of the trajectory (\ref{eqn:alpha})
with initial condition $z_0$ up to time $t$ is an easy (freshman
calculus) integral to compute from Equation (\ref{eqn:cost-upper}),
which we do not display here.  The total cost of
$\gamma(\kappa,z_0,t)$ from time $0$ to $T_n$ is
\begin{equation}\label{eqn:total-cost}
\sum_{i=0}^{n-1} \op{cost}(z_i,[0,t_{i+1}]).
\end{equation}

\subsection{the smoothed regular polygon}\label{sec:polygon}

Reinhardt conjectured that the smoothed octagon is the solution
to his problem.

The smoothed octagon comes from a periodic bang-bang control to the
state equations with four links (and four switching times).  The
control switches four times in a cyclic order around the extreme
points of the control simplex $U$.  The smoothed octagon itself can be
visualized as being made of $24$ segments: $8$ smoothed corners and
$16$ half-edges.  These $24$ segments are arranged into four links,
each consisting of $6$ arcs.  The four links are congruent, under the
rotational symmetry $R$.

We generalize the smoothed octagon to a smoothed regular polygon as
follows.  Let $k$ be a positive integer.  We consider a trajectory
with $3k+1$-links of the same length of the form
$t\mapsto\gamma(\kappa,z_k,t)$, with
\begin{equation}\label{eqn:3k+1}
\kappa = ((0,t_k),(-1,t_k),(-2,t_k),\ldots,(-3k,t_k)),
\end{equation}
where $t_k>0$ and $z_k\in\h$ are to be determined as functions of $k\ge1$.  
(Note that the meaning of $t_k,z_k$ has changed from $t_i,z_i$ in the previous
section.)

Let $g_k = \gamma(z_k,t_k)\in \SL$ be the position at the end of a
single link.  The endpoint condition (\ref{eqn:g-term}) for
(\ref{eqn:3k+1}) is
\[
R = g_k (R^{-1} g_k R^1) (R^{-2} g_k R^2) \cdots (R^{-3k} g_k R^{3k}),
\]
or equivalently,
\begin{equation}\label{eqn:eigen}
(R^{-1} g_k)^{3k+1} = R^{-3k} = (-I)^k.
\end{equation}
Let $\mu,\mu^{-1}$ be the eigenvalues of $R^{-1} g_k\in \SL$.
Comparing eigenvalues on the two sides of (\ref{eqn:eigen}), we obtain
$\mu^{3k+1} = (-1)^k$, and
\[
\mu = e^{\pi i k/(3k+1) + 2\pi i \ell/(3k+1)},\quad \ell \in\ring{Z}.
\]
We pick the eigenvalues $\mu^{\pm1}$ that place $g_k$ in the smallest
neighborhood of of $1$; that is, we take $\ell=0,-k$.  Then
\begin{equation}\label{eqn:theta-k}
\op{trace}(R^{-1} g_k) = \mu+\mu^{-1}=2\cos\theta_k,\quad 
\text{where } \theta_k= \frac{\pi k}{3k+1}.
\end{equation}
For example, for the smoothed octagon $k=1$, the trace is
$\sqrt{2}$.

We impose the strong boundary condition
\begin{equation}\label{eqn:poly-z}
z(z_k,t_k) = R^{-1}.z_k,\quad \text{where } z_k = 0 + i y_k.
\end{equation}
It follows that (\ref{eqn:z-term}) holds with $t_f = (3k+1) t_k$:
\[
z(z_k,t_f) = R^{-(3k+1)}.z_k = R^{-1}.z_k.
\]
Solving (\ref{eqn:poly-z}) for $t_k$ (the time spent in each link), we
obtain
\begin{equation}\label{eqn:t1}
t_k = \frac{|\ln ((1+ 3y_k^2)/4)|}{\sqrt{3} y_k} 
\end{equation}
We have solved the nonlinear equations (\ref{eqn:theta-k}) and
(\ref{eqn:t1}) explicitly for $t_k$ and $y_k$ in the accompanying
code, but we do not display the solution here.  For each positive
integer $k$, the trajectory for the smoothed $6k+2$-gon is now
completely determined by these values of $t_k$ and $y_k$.

These formulas for $t_k$ and $y_k$ can be interpolated to functions
that are analytic in $k\in \ring{R}$.  Figure \ref{fig:cost} graphs the area of the
smoothed $6k+2$-gon as a function of $k$.  It appears that the area
function is increasing in $k$ and tends to the area $\pi$ of a circular disk.

We show that we can lift each trajectory to a Pontryagin extremal.
The following is one of the main conclusions of this article.  It
implies in particular that the smoothed octagon $k=1$ is a Pontryagin
extremal.

\begin{theorem}\label{thm:pmp} 
The smoothed regular $6k+2$-gon lifts to a Pontryagin
  extremal trajectory.  The trajectory has a normal multiplier.
\end{theorem}

\begin{proof}
We show that there exists a choice of initial conditions for
$\Lambda,\nu$ for which the PMP conditions hold.

We start with the endpoint condition (\ref{eqn:transverse}) for $\Lambda$.
Again, we prove a stronger form of transversality by showing
\begin{equation}\label{eqn:trans-Lambda}
\Lambda(t_k) = R^{-1} \Lambda(0) R,
\end{equation}
which implies (\ref{eqn:transverse}).  This is a system of three
homogeneous equations for $\Lambda(0)\in \sl$ (three unknowns).  We
indicate why a nontrivial solution to this homogeneous system for
$\Lambda(0)$ must exist.  By the form of the differential equation it
satisfies (\ref{eqn:adjoint}), as $\Lambda$ evolves in time, its
determinant remains constant.  We find that $\Lambda(t)$ remains in a
fixed conjugacy class of $\sl$.  We can therefore write $\Lambda(t) =
h(t) \Lambda(0) h(t)^{-1}$ for some $h(t)\in\SL$.  Equation
(\ref{eqn:trans-Lambda}) asserts that $\Lambda(0)$ lies in the
centralizer of $R h(t_k)$ in $\sl$.  A centralizer has minimal
dimension $1$ (which occurs when $R h(t_k)$ is regular, which occurs
here).  Thus, solutions exist and are unique up to a scalar.

We have a linear system of five equations and five unknowns.  The five
unknowns are $\Lambda(0)\in\sl$ and $\nu(0)\in\ring{R}^2$.  Two
independent equations come from (\ref{eqn:trans-Lambda}), one from the
vanishing of the Hamiltonian, and two from the endpoint condition
(\ref{eqn:transverse}) on $\nu$ (for reduced period $t_k$ instead of
$t_f$).  Explicit calculations give a unique solution to this linear
system of equations (as homogeneous functions of $\lambda_{cost}$) for
each $k$.  This forces the multiplier $\lambda_{cost}$ to be normal,
and we take $\lambda_{cost}=-1$.

Further explicit symbolic computer-algebra calculations show that
$t=0$ and $t=t_k$ are switching times.
The following lemma completes the proof, which shows that the
maximum property of Pontryagin is met for the Hamiltonian.
\end{proof}

\begin{lemma}  
  Let $k\ge 1$ be an integer.  Let $\lambda(k,t)$ be the lifted
  trajectory for the smoothed $6k+2$-gon along a single link with control
  $u=e_3$ as constructed above.  Let $H_{k,u}(t)$ be the
  Hamiltonian restricted to the lifted trajectory, with arbitrary control
  function $u:[0,t_k]\to U$.  Then
\[
H_{k,e_3}(t) \ge H_{k,u(t)}(t),\quad t\in [0,t_k].
\]
If $u\in \{e_1,e_2\}$ is a constant control at one of the first two vertices
of $U$, then equality occurs only at the endpoints of the interval $[0,t_k]$.
\end{lemma}

\begin{proof}
  A monotonicity result  (Section \ref{sec:bang}) shows
  that the maximum of $H_{k,u}(t)$ is attained at a corner of the control simplex.  It
  is enough to show that $ \chi_{k,u}(t)\ge 0$, where
  $\chi_{k,u}=H_{k,e_3}-H_{k,u}$, for the two constant controls
  $u=e_1$ and $u=e_2$.  

An easy substitution using
  the explicit formulas for $\lambda(k,t)$ gives
\[
\chi_{k,e_1}(t_k-t) = \chi_{k,e_2}(t),\quad t\in[0,t_k].
\]
Thus, it is enough to show that $\chi_{k,e_2}(t)\ge 0$.  The function
$\chi_{k,e_2}(t)$ is equal to $\nu(k,t)_2$ up to a positive
nonzero factor.  Thus, the lemma reduces to proving that
$\nu(k,t)_2\ge0$ for $k\ge 1$.  (Here $\nu(k,t)_2$ is
the component $\nu_2$ of Section \ref{sec:ham} of the lifted trajectory
$t \mapsto \lambda(k,t)$.)

We define new variables $(y,v)$:
\[
y = 1+ 3y_k^2,\quad v = y e^{\sqrt{3} y_k t},
\]
and replace $k$ with a continuous parameter.  The region defined
by $k\ge1$ and $t\in[0,t_k]$ transforms to the triangle 
\[
T=\{(y,v)\in [2\sqrt{2},4]^2\mid  y\le v\}.
\]
Note that $t=0$ is transformed to the diagonal $y=v$ of $T$.
We define
\[
f(y,v) = 3\sqrt{3} y_0^5 v^2 \nu(k,t)_2.
\]

We show that $f$ is nonnegative on the triangle $T$ as follows.
(These calculations appear in the accompanying computer  code.)
First, an easy substitution gives $f(y,y)=0$.  (This was already
verified above in a different manner, when we showed that $t=0$ is a
switching time.)  Second, the derivative is negative on the diagonal:
\begin{equation}\label{eqn:f'}
\partials {f}{y}|_{v=y} = y((y-4) + y \ln (4/y)) \le 0.
\end{equation}
Finally, the second derivative is positive on $T$:
\[
\frac{\partial^2 f}{\partial y^2} = -10 - 5v + 2v/y + 7y - 
 2v\ln 4 + 6y\ln 4  
+ 4v\ln v  - 2(v + 3y)\ln y\ge 0.
\]
(We leave this last inequality as a tedious but elementary exercise for the reader.)
Positivity follows.

Looking more closely at the cases of equality, we see that the only
zero of the switching function on $[0,t_k]$ occurs at $t=0$, and that
the derivative is strictly positive at $t=0$.  (The derivative is zero
in (\ref{eqn:f'}) at the corner $v=y=4$ of the disk, but this
corresponds to the unrealizable limiting case as $k\mapsto\infty$.)
\end{proof}

\begin{remark}
A related constructed gives a trajectory with $3k-1$-links -- the
smoothed $6k-2$-gon $D_{6k-2}$, for $k\ge 2$. 
The changes are minor.  We replace equation
(\ref{eqn:3k+1}) with
\begin{equation}\label{eqn:3k+2}
\kappa = ((1,t_k),(2,t_k),(3,t_k),\ldots,(3k-1,t_k)).
\end{equation}
The trajectory is
\[
\gamma(\kappa,R^{-1}.z_k,t).
\]
Equation \ref{eqn:theta-k} becomes
\begin{equation}\label{eqn:theta-k'}
\op{trace}(R g_k) = 2\cos\theta_k,\quad 
\text{where } \theta_k= \frac{\pi k}{3k-2}.
\end{equation}
Equation (\ref{eqn:t1}) is unchanged.  The
initial link of the smoothed octagon now has constant control
$u=e_2$.
\end{remark}

\begin{remark} It seems that the smoothed $6k-2$-gon $D_{6k-2}$ is {\it
    not} a Pontryagin extremal trajectory.  Specifically, all of the
  conditions seem to hold, except that the Pontryagin multiplier
  $\lambda_{cost} >0$ has the wrong sign.  This suggests that these
  smoothed polygons are Pontryagin extremal trajectories for the
  problem of maximizing the area.
\end{remark}

\begin{remark}
  When $k=1$, the smoothed polygon $D_4$ degenerates to a rectangle
  with corners (Figure \ref{fig:rectangle}) and area $\sqrt{12}$.  Allowing
  $k$ to be non-integral, for small values of $k>1$, we obtained
  smoothed rectangles (that do not quite satisfy the boundary conditions).
\end{remark}

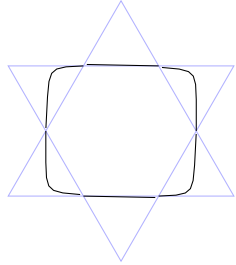
\tikzfig{rectangle}{ By taking a smoothed $6k-2$-gon and interpolating
  formulas to a fractional number of sides (here $k=1.03$), we see
  that the shape appears to be tending to a rectangle of area
  $\sqrt{12}$ as $k\mapsto 1$.  }{
\def\rt{1.732};
\draw[blue!30] (0,-\rt) -- (1.5,\rt/2) -- (-1.5,\rt/2) -- cycle;
\draw[blue!30] (0,\rt)-- (-1.5,-\rt/2) -- (1.5,-\rt/2) --cycle;
\draw 
(1., 0.)-- (1., 0.0169549)-- (1., 0.048467)-- (1., 0.107035)-- (1., 
  0.215889)-- (1., 0.418202)-- (0.991058, 0.620692)-- (0.960685, 
  0.730146)-- (0.896834, 0.789974)-- (0.77418, 0.823909)-- (0.544074, 
  0.845408);
\draw (0.5, 0.866025) --  (0.269022, 0.870587) --  (0.144746, 
  0.873041) --  (0.0778794, 0.874361) --  (0.0419025, 0.875072) --  (0.0225454,
   0.875454) --  (0.00318826, 0.875836) --  (-0.0327886, 
  0.876547) --  (-0.0996547, 0.877867) --  (-0.223931, 
  0.880321) --  (-0.454909, 0.884882);
\draw (-0.5, 0.866025) --  (-0.730978, 0.853632) --  (-0.855254, 
  0.824574) --  (-0.922121, 0.767326) --  (-0.958097, 0.659183) --  (-0.977455,
   0.457252) --  (-0.98787, 0.255144) --  (-0.993473, 0.146401) --  (-0.996488,
   0.0878927) --  (-0.998111, 0.0564126) --  (-0.998983, 0.0394749);
\draw (-1., 0.) --  (-1., -0.0169549) --  (-1., -0.048467) --  (-1., -0.107035) --   
(-1., -0.215889) --  (-1., -0.418202) --  (-0.991058, -0.620692) --   
(-0.960685, -0.730146) --  (-0.896834, -0.789974) --  (-0.77418,  
-0.823909) --  (-0.544074, -0.845408);
\draw (-0.5, -0.866025) --  (-0.269022, -0.870587) --  (-0.144746, -0.873041) --   
(-0.0778794, -0.874361) --  (-0.0419025, -0.875072) --  (-0.0225454,  
-0.875454) --  (-0.00318826, -0.875836) --  (0.0327886, -0.876547) --   
(0.0996547, -0.877867) --  (0.223931, -0.880321) --  (0.454909, -0.884882);
\draw (0.5, -0.866025) --  (0.730978, -0.853632) --  (0.855254, -0.824574) --   
(0.922121, -0.767326) --  (0.958097, -0.659183) --  (0.977455, -0.457252) --   
(0.98787, -0.255144) --  (0.993473, -0.146401) --  (0.996488, -0.0878927) --   
(0.998111, -0.0564126) --  (0.998983, -0.0394749);
}

The trajectory in $\h$ for $D_{6k+2}$ follows a triangle (with edges
following the arcs of Figure \ref{fig:traj-upper}) centered at
$z=i\in\h$.  It moves counterclockwise around $i$, traversing one edge
for each link (Figure \ref{fig:tri}).  The trajectory in $\h$ for
$D_{6k-2}$ also follows an inverted triangle
centered at $z=i\in \h$.  It moves clockwise.

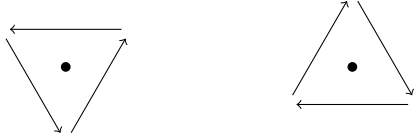
\tikzfig{tri}{The trajectory in the upper-half plane of a smoothed
  $6k+2$-gon follows $3k+1$ edges moving counterclockwise on a
  triangular path centered at $i\in\h$ (left).  The trajectory for the
  smoothed $6k-2$-gon follows $3k-1$ edges moving clockwise on an
  inverted triangle centered at $i\in\h$ (right).}  {
\def\rt{1.732}
\begin{scope}
\node (a) at (0,-1) {};
\node (b) at (\rt/2,0.5) {};
\node (c) at (-\rt/2,0.5) {};
\draw[->] (a) to node {} (b);
\draw[->] (b) to node {} (c);
\draw[->] (c) to node {} (a);
\smalldot{0,0};
\end{scope}
\begin{scope}[xshift=1.5in]
\node (a) at (0,1) {};
\node (b) at (-\rt/2,-0.5) {};
\node (c) at (\rt/2,-0.5) {};
\draw[<-] (a) to node {} (b);
\draw[<-] (b) to node {} (c);
\draw[<-] (c) to node {} (a);
\smalldot{0,0};
\end{scope}
}

The cost increases with $k$ for $D_{6k+2}$ and decreases with $k$ for
$D_{6k-2}$.  In both cases, the limit of the cost is $\pi$ as
$k\mapsto\infty$.  We show a graph of the costs of the smoothed
polygons as a function of the number $n=6k\pm 2$ of sides (Figure
\ref{fig:cost}). 

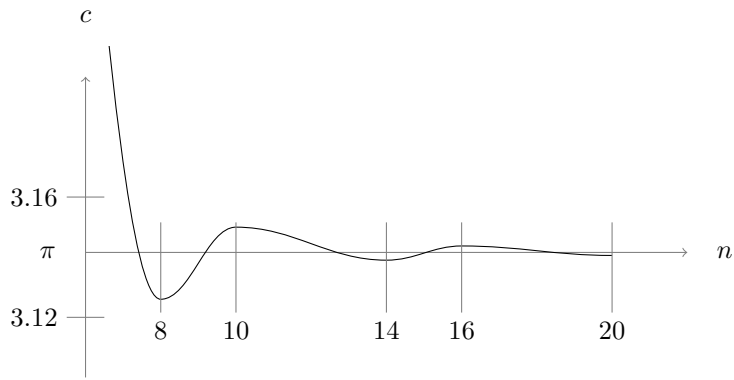
\tikzfig{cost}{The graph interpolates the cost $c$ of known critical
  points as a function of the number $n=6k\pm 2$ of straight edge
  segments in the corresponding smoothed polygon.  The cost tends to
  $\pi$ as $n$ increases.  The data is consistent with Reinhardt's conjecture.}{
\begin{scope}[xscale=0.5,yscale=40]
\draw[->,gray] (6,3.14159) to (22,3.14159);
\draw[->,gray] (6,3.1) to (6,3.2);
\node (n) at (23,3.14159) {$n$};
\node (pi) at (5,3.14159) {$\pi$};
\node (c) at (6,3.22) {$c$};
\clip (4,3.1) rectangle (22,3.21);
\draw  (4,3.456) to [out=0,in=180] (8,3.126);
\draw (8,3.126) to [out=0,in=180] (10,3.15);
\draw (10,3.15) to [out=0,in=180] (14,3.139);
\draw (14,3.139) to [out=0,in=180] (16,3.14374);
\draw (16,3.14374) to [out=0,in=180] (20,3.14058);
\foreach \x in {8,10,14,16,20} \draw[gray] (\x,3.14159+0.01) -- 
(\x,3.14159 -0.02) node[anchor=north,black] {$\x$};
\foreach \y in {3.12,3.16} \draw[gray] (6+0.5,\y) -- 
(5.5,\y) node[anchor=east,black] {$\y$};
\end{scope}
}

\subsection{(micro) local optimality of the smoothed octagon}

Nazarov has proved that the smoothed octagon is a local minimum of the
Reinhardt problem \cite{nazarov1988reinhardt}.  The following theorem
should be viewed as a control-theory analogue of Nazarov's theorem.
Our result gives {\it micro-local} optimality in the sense that we
consider a neighborhood $V$ of the lifted extremal trajectory in the
cotangent space.  The following is one of the main results of this
article.

\begin{theorem}\label{thm:local-min} Let $\lambda_{oct}:[0,t_f]\to T^*M$
  be the Pontryagin extremal lifted trajectory constructed in the
  previous section for the smoothed octagon ($k=1$).  Then
\begin{enumerate}
\item there exists a punctured neighborhood $V^*$ of
  $\lambda_{oct}(0)\in T_I^*\Mstar$
  such that no initial condition in $V^*$ gives a Pontryagin extremal
  lifted trajectory.
\item Moreover, if any initial condition in $V^*$ gives a hexagonally
  symmetric disk $D$, then the area of $D$ is greater than that of the
  smoothed octagon.
\end{enumerate}
\end{theorem}

\begin{remark}
I have not checked whether the other smoothed $6k+2$-gons are
local minima in the same sense.
\end{remark}

\begin{proof}[Proof (sketch)] From the explicit form of the lifted
  trajectory $\lambda_{oct}$ that was constructed in the previous
  section, we see that the switching functions meet the $x$-axis
  transversally at $0$ and $t_1$ and have no other zeros on the
  interval $[0,t_1]$.  Thus, any sufficiently small perturbation of
  the initial conditions will produce a small perturbation of the
  switching times.  In particular, the trajectory will continue to
  consist of four links of approximately the same size and with the
  same controls as before on each link.  We can assume without loss of
  generality that $t=0$ is a switching time.

Thus, we can write the perturbed state in the form
$t\mapsto\gamma(\kappa,z,t)$, where 
\[
\kappa = ((0,t_1+\eta_1),(-1,t_1+\eta_2),(-2,t_1+\eta_3),(-3,t_1+\eta_4)),
\quad \eta = (\eta_1,\eta_2,\eta_3,\eta_4,\eta_5,\eta_6)
\]
and where $z= iy_1 + \eta_5+i \eta_6\in\h$ lies in a small neighborhood of
$0+i y_1\in\h$, and $\eta_i\in\ring{R}$ are near $0$.  Here,
$(t_1,y_1)=(t_k,y_k)$, constructed in Section \ref{sec:polygon}, with $k=1$.

We prove the second claim of the theorem first.  We have a
six-dimensional parameter space of initial conditions $\eta\in
\ring{R}^6$ and five endpoint equations (\ref{eqn:g-term}) and
(\ref{eqn:z-term}) (counting three equations from $\SL$ and two from
$\h$).  These equations define a one-dimensional curve $N\subset
\ring{R}^6$ through $p=0$.  The curve represents a $1$-dimension
family of deformations of the smoothed octagon that satisfies the
endpoint conditions.  These equations satisfy the conditions of the
analytic implicit function theorem, allowing us to use $\eta_1$ as an
analytic coordinate on $N$ near $p$.  We write the other coordinates
$\eta_2,\ldots,\eta_6$ as power series in $\eta_1$ on $N$ near $p$:
\begin{equation}\label{eqn:h}
\bar\eta_j := \eta_j|_N  = \bar\eta_1 a_j^1 + 
\bar\eta_1^2 a_j^2 + O(\bar\eta_1^3).
\end{equation}
for some  coefficients $a_j^1,a_j^2\in\ring{R}$ to be
calculated.  Initial calculations show that the constants $a_j^1$
have the form
\[
a_1^1=1,\quad a_2^1=-1,\quad a_3^1=1,\quad a_4^1=-1.
\]
The choice of local parameter $\bar\eta_1$ on $N$
gives $a_1^2=0$.  The terminal
time for the deformation is $t_f(\bar\eta_1) = t_f +
\bar\eta_1+\bar\eta_2+\bar\eta_3+\bar\eta_4$, where $t_f=4t_1$ is the terminal time
for the smoothed octagon.  We write the periodic endpoint conditions
(\ref{eqn:g-term}) (\ref{eqn:z-term}) in the form
\begin{align}\label{eqn:gh}
\begin{split}
g(\bar\eta,t_f(\bar\eta_1)) &= R;\\
z(\bar\eta,t_f(\bar\eta_1)) &= R^{-1}.(i y_1 + (\bar\eta_5+i\bar \eta_6)).
\end{split}
\end{align}
A long computer algebra calculation (using interval arithmetic, automatic
differentiation, these endpoint conditions, and the explicit formulas
for the solutions to our ODEs) gives us the power series expansion of
the left-hand side of (\ref{eqn:gh}) to second order in terms of the
unknown coefficients $a_j^i$.  This is a delicate calculation,
which explicitly propagates the unknown coefficients along the trajectory
to the endpoint. Comparing with the right-hand side of (\ref{eqn:gh}),
we obtain explicit interval arithmetic bounds on $a_j^1,a_j^2$.  Still
using computer algebra calculations, we write the
cost over the time interval $[0,t_f(\bar\eta_1)]$
as a function of $\bar\eta_1$ on $N$, and expand the cost in a power
series in $\bar\eta_1$ using these interval bounds on $a_j^1,a_j^2$.
These interval arithmetic calculations give the explicit bounds
\begin{equation}\label{eqn:local-min}
\op{cost}(\bar\eta_1) = 
\op{cost}(0) + b_1 \bar\eta_1 + b_2 \bar\eta_1^2 + O(\bar\eta_1^3),
\end{equation}
where $|b_1|< 10^{-9}$ and $b_2= 4.7976\ldots$.  
In particular $\op{cost}''(0)>0$.

We know that the smoothed octagon is a Pontryagin extremal.  This
implies that no needle perturbations of the smoothed octagon can give
a first-order improvement to the cost.  In particular, $b_1=0$.  Thus,
by (\ref{eqn:local-min}), cost has a strict local minimum at
$\bar\eta_1=0$.  This completes the proof of the final claim of the
theorem.

Finally, we give a proof of the first claim of the theorem: there
exists $V^*$ such that no initial condition in $V^*$ gives a
Pontryagin extremal lifted trajectory.  Any such lifted trajectory
satisfies the endpoint conditions, and must therefore have an initial
condition of the form (\ref{eqn:h}) and must lie in $N$.  On some punctured neighborhood
of the smoothed octagon, along $N$, the cost
(\ref{eqn:local-min}) has nonzero derivative.  This is inconsistent
with PMP. This completes the proof.
\end{proof}

\section{The singular locus}
\label{sec:sing}

\subsection{the circle as singular arc}\label{sec:circle}

The circular disk is a hexagonally symmetric disk $D$ defined by the
trajectory
\begin{equation}\label{eqn:circle}
g(t) = \exp(J t),
\quad g' = g J,\quad t\in [0,t_f],\quad t_f=\pi/3.
\end{equation}
(The six paths $t\mapsto \sigma_j(t)=g(t) \ee{j}$, for $t\in[0,\pi/3]$
are six arcs that fill out the unit circle.)  Thus, $X\equiv J$ is a
constant path, and $x+iy = i$ is also constant.  In the Poincar\'e
disk, the constant path is $w\equiv 0$.  The cost from Equation
(\ref{eqn:cost-disk}) is
\[
3 \int_0^{\pi/3} dt = \pi,
\]
which has the expected interpretation of the area of
a circular disk of radius $1$.

The control for the circle is constant:
$u=(1/3,1/3,1/3)\in U$.  The ODE (\ref{eqn:ode-state-upper})
becomes
\begin{align*}
x'& = \frac{y (1+ x^2-y^2)}{1+x^2+y^2}\\
y'& = \frac{-2 x y^2}{1+x^2+y^2},
\end{align*}
which indeed has the constant solution $x\equiv0$, $y\equiv1$.  This
constant solution determines the constant path $X\equiv J\in \sl$.

In our context, a lifted trajectory $\lambda$ is a {\it singular arc}
on $[t_1,t_2]$ if for all $t\in[t_1,t_2]$, the face
$U_{\lambda(t)}\subseteq U$ has positive dimension.  See
\cite[\S4.4.3]{liberzon2012calculus}.

\begin{lemma} The circle is an extremal singular arc.  The multiplier
  is normal.
\end{lemma}

\begin{proof}  The solution to the adjoint equations is also constant:
\begin{equation}\label{eqn:adj-circle}
\Lambda \equiv \Lambda_0=\frac{3}{2} J\lambda_{cost},
\quad \nu\equiv 0.
\end{equation}
where $\lambda_{cost}=-1$ (for a normal multiplier).  A simple calculation
based on this explicit data shows the circle is an extremal.  Along
the lifted trajectory, the Hamiltonian is independent of the control:
\[
H(\lambda(t);u)\equiv 0.
\]
Thus, $U_{\lambda(t)}=U$ and the lifted trajectory is a singular arc.
\end{proof}

\begin{remark}
  Second order conditions show that circular arc is not a local
  minimizer on any time interval $[t_1,t_2]$ so that the solution to
  the Reinhardt problem contains no circular arcs
  \cite[\S5.2]{hales2011reinhardt}.  We recall the argument.
  We consider a deformation of
  a circular arc of the form
\[
g_\epsilon(t) = \exp\left(\epsilon 
\begin{pmatrix}c_{11}(t) & c_{12}(t)\\ c_{12}(t) & -c_{11}(t)
\end{pmatrix}
\right) e^{J t}
\]
for sufficiently small $\epsilon>0$ and compactly supported $C^\infty$
functions $c_{11}$, $c_{12}$ to be determined on the interval
$[t_1,t_2]$.  We emphasize that $t$ is not a unit speed parameter.
Computing the cost of $g_\epsilon$ on $[t_1,t_2]$ by
(\ref{eqn:cost-JX}), we find that
\[
\op{cost}(g_\epsilon) = \op{cost}(g_0) + 
6 \epsilon^2 \int_{t_1}^{t_2} c_{11}(t) c_{12}'(t)\,dt + O(\epsilon^3).
\]
Note that this is a second variation that is not detected by PMP.
Choose $c_{11}(t)\ge0 $ (with positive integral $\int c_{11} dt>0$)
with support on an interval where $c_{12}'(t)<0$.  Then for all
sufficiently small $\epsilon>0$, we have
\[
\op{cost}(g_\epsilon) < \op{cost}(g_0)=\pi.
\]
We may pick $\epsilon>0$ sufficiently small so that the curvatures of
the curves $t\mapsto g_\epsilon(t)\ee{i}$ are positive.  Then there
exists a control function $u:[t_1,t_2]\to U$ with controlled
trajectory $g_\epsilon$.
\end{remark}

\subsection{no singular arcs}

Recall that a Reinhardt trajectory is a trajectory $(g,z)$ such
that its disk $D(g,z)=\DR$ is a globally optimal solution to the
Reinhardt problem.

\begin{lemma} A Reinhardt trajectory contains no
singular arcs.
\end{lemma}

\begin{proof} Along a singular arc, the set $U_\lambda$ 
of controls maximizing the Hamiltonian has positive dimension.
The set can be an edge of $U$ or all of $U$.
We first assume that $U_{\lambda(t)}$ is an edge on a set 
of positive measure.
By the continuity of the lifted singular arc,  $U_{\lambda(t)}$
is a fixed edge on an open set in $[t_1,t_2]$.
We show that this leads to a contradiction.

By symmetry, without loss of generality, we may assume
that the endpoints of the edge are $e_2,e_3\in U$.
Thus, the first component of the control $u$ is identically zero along
the edge.  That is, $u=(0,*,*)$ along the singular arc.  Interpreting
the vanishing of the first component of $u$ geometrically as a zero
planar-curvature constraint (\ref{eqn:kappa}), the equation
(\ref{eqn:sigma}) implies that the path $\sigma_0$ traces out a line
in $\ring{R}^2$.  After applying an affine transformation to make this
line horizontal, we may assume that $\sigma_0$ has the form
\[
\quad \sigma_0(t) = (\xi(t),-1) = 
\begin{pmatrix} 1 & -\xi(t) \\ 0 &1\end{pmatrix} 
\begin{pmatrix} 0\\1\end{pmatrix},
\]
for some function $\xi:[t_0,t_1]\to \ring{R}$.  Recall that
$\sigma_0(t) \wedge \sigma_2(t) = \sqrt{3}/2$ (see
\cite[\S3.2]{hales2011reinhardt}).  This implies
that $\sigma_2$ has the form
\[
\sigma_2(t) = \begin{pmatrix} 1 & -\xi(t) \\ 0 &1\end{pmatrix} 
\begin{pmatrix} -\sqrt{3}/2 \\ (1+s(t))/2\end{pmatrix}
\]
for some function $s:[t_0,t_1]\to \ring{R}$.

Rather than using the unit speed normalization from (\ref{eqn:unit}),
it is more convenient to choose a linear parameter such that $\xi(t) =
a_0 t + b_0$.  This requires us to make a few minor adjustments to the
optimal control problem that are adapted to the singular arc.  By
picking the parameters $a_0,b_0$ suitably, we can assume that
$\sigma_0$ starts at time $t=0$ and reaches its terminal position on
the singular arc at time $t=1$.  We optimize among trajectories with
fixed initial and terminal positions: $(r(0),s(0))=(r_0,s_0)$,
$(r(1),s(1))=(r_1,s_1)$, where $r:= s'$.

There is a unique  $g:[0,1]\to \SL$ such that
(\ref{eqn:sigma}) holds.  We compute
\[
g(t) = \begin{pmatrix}
b_0 + a_0 t & (\sqrt{3}-(a_0 t+ b_0) s)/{\sqrt{3}}\\
-1 & s/{\sqrt{3}}
\end{pmatrix}.
\]
Defining $X$ by (\ref{eqn:g}), we describe the state by the pair of
functions $(r,s)$.  (Crucially, unlike the treatment above, the
function $g$ is not included in the state.  The terminal condition for
$g$ is already determined by the terminal condition $s_1$ of the
functions $s$.)  The state equations are
\begin{align*}
s' &=  r,\\
r' &= \frac{2 (-1 + 2 u) r^2}{-1 - s + 2 u s}.
\end{align*}
The control is now $u\in [0,1]$ (representing an edge of the earlier
control simplex $U$).

Without normalizing to unit speed,
the star inequality gives
\begin{equation}\label{eqn:star-sing}
\det(X) = \frac{a_0 r}{\sqrt{3}}>0, \quad c_{21}(X) = a_0 > 0.
\end{equation}

The cost functional is 
\begin{align*}
-\frac{3}{2}\int_0^1 \op{trace}(J X) dt &= 
\frac{1}{2}\int_0^1 (3 a_0 +   \sqrt{3} s' + a_0 s^2) dt \\
&= 
\frac{1}{2} (3 a_0 + \sqrt{3}(s_1-s_0)) + \frac{a_0}{2} \int_0^1 s^2 dt\\
&= C_1 + \frac{a_0}{2} \int_0^1 s^2 dt.
\end{align*}
We drop the useless constant $C_1$ from the cost and form the
Hamiltonian
\[
\nu_1 r + \nu_2 
\frac{2 (-1 + 2 u) r^2}{-1 - s + 2 u s} + \frac{a_0\lambda_{cost} s^2}{2}.
\]
(The Lie group term is no longer present.)  The condition for the
Hamiltonian to be independent of $u$ is $\nu_2=0$.  Thus, $\nu_2
\equiv 0$ along the singular arc.  Solving the adjoint equations, we
get $\nu\equiv 0$ along the singular arc.  The nonvanishing of the
costate gives $\lambda_{cost}\ne 0$.  The Hamiltonian reduces to
$a_0\lambda_{cost} s^2/2$, which must be constant.  Hence $s$ is
constant, and $s'=r=0$, which contradicts the star inequality
(\ref{eqn:star-sing}).  Hence, no singular arc exists in this case.
We have completed the proof when $U_{\lambda(t)}$ is an edge for $t$
in some time interval.

In the remaining case, $U_\lambda = U$ on some time interval.
This implies that $\nu \equiv 0$
along the singular arc.  The adjoint equations and PMP imply that for
all $t$ along the singular arc,
\[
H =  -\partials {H}{ x} = -\partials {H}{y} =0.
\]
Solving these equations for $\Lambda$, we find a unique solution
\[
\Lambda(t) \equiv \frac{3}{2} J \lambda_{cost}.
\]
The adjoint equation $\Lambda' = [\Lambda,X]\equiv 0$ implies that
$x\equiv 0$, $y\equiv 1$.  This is the equation of a circle, which we
have seen is a singular arc.  As remarked above, a circular arc is not
second-order optimal and does not occur in the optimal solution to the
Reinhardt problem.
This completes the proof.
\end{proof}

\subsection{switching functions}

We define switching functions $\chi_{ij}:T^*\Mstar\to \ring{R}$ by
\[
\chi_{ij} (\lambda) = H(\lambda,e_i) - H(\lambda,e_j).
\]
The optimal control is constant $u=e_i$ (that is, $U_\lambda=\{e_i\}$), 
on parts of the cotangent
space where $\chi_{ij}>0$ for all $j\ne i$.  For example,
\[
\chi_{32} = 2\sqrt{3} \nu_2 y^2/(1- 3 x^2),
\]
which equals $\nu_2$, up to a positive factor.

Let $\lambda_{cost}=-1$ and
$\lambda_{sing,g}\in T^*M$
be the initial conditions (\ref{eqn:adj-circle}) matching the circle:
\begin{align*}
\lambda_{sing,g}&=(\Lambda_{sing},\nu_{sing}) = \left(
\frac{3}{2} J\lambda_{cost},0\right)\in T_{g,z_0}^*M,
\\
z_0&= (x_0,y_0) = (0,1)\in\h.
\quad g\in\SL.
\end{align*}
We define the {\it singular locus} $\Lsing$ of $\{-1\}\times T^*M$ by
\begin{equation}
\Lsing= \{\lambda_{cost}\}\times \{\lambda_{sing,g}\mid g\in \SL\}\subset
 T^*M.
\end{equation}
We write $\lsing=\lambda_{sing,I}\in T_I^*M$.  We note that up to the
affine transformation $g\in \SL$, the singular locus is the initial
condition $\lsing\in T_I^*M$ defining the singular circular arc; that
is, $g_{\lsing}(0)=I$.

We show that no transition is possible between a Pontryagin extremal
link and a circular arc.

\begin{lemma} There does not exist a Pontryagin extremal link with
  constant control $u\in\{e_1,e_2,e_3\}$ with initial conditions (or terminal conditions)
$\lambda_{sing,g}$ and $\lambda_{cost}=-1$. 
\end{lemma}

(This lemma does not rule out the possibility of a chattering arc
meeting a singular arc \cite[Fig.20.1]{agrachev2013control}.  See
below.)

\begin{proof} 
By symmetry, we may assume that the link (if
it exists) has constant control $u=e_3$.
If we take the general solution to the adjoint equation
  with control $u=e_3$ and match it with the given initial
  conditions with normal multiplier $\lambda_{cost}=-1$, we compute that
\[
\nu_2(t) = \frac{-1 + e^{-2\sqrt{3} t} + 2 \sqrt{3} t e^{-\sqrt{3} t}}{\sqrt{3}}
\]
This function is easily checked to be negative for all $t>0$.  Recall
that $\nu_2$ has the same sign as the switching function between
controls $u=e_2$ and $u=e_3$.  PMP requires $\nu_2(t)$ to
be positive when the control is $u=e_3$.  Thus, a link that
matches initial conditions with the circle cannot be a Pontryagin
extremal.
\end{proof}


\subsection{finiteness of switching}

We need the  following simple lemma in preparation for the main
theorem (\ref{thm:finite}) of this section.

\begin{lemma}\label{lemma:sl2-basis}
Let $X:\h\to\sl$ be given by Equation (\ref{eqn:X-h}).
For all $x+iy\in\h$,
\begin{equation}\label{eqn:sl2-basis}
X,~\partials {X}{ x},~\partials {X}{ y}
\end{equation}
gives a basis of $\sl$.
\end{lemma}

\begin{proof}
Let $L$ be a linear transformation that sends the
the standard basis:
\[
\begin{pmatrix}1&0\\0&-1\end{pmatrix},
\quad
\begin{pmatrix}0&1\\0&0\end{pmatrix},
\quad
\begin{pmatrix}0&0\\1&0\end{pmatrix}.
\]
to the three vectors (\ref{eqn:sl2-basis}).  
The absolute value of the determinant of the linear
transformation $L$ is $2/y^2\ne0$.
\end{proof}

\begin{theorem}\label{thm:finite}
  Let $\lambda:[0,t_f]\to T^*M$ be a Pontryagin extremal that does not
  meet the singular locus $\Lsing$.  Then $\lambda$ has a bang-bang
  control with finitely many switches.
\end{theorem}

\begin{remark} In terms of Reinhardt's problem, the theorem implies
  that an extremal trajectory $\lambda$ that does not meet the
  singular locus $\Lsing$ defines a hexagonally-symmetric disk
  $D(g_\lambda,z_\lambda)$ whose boundary is a smoothed polygon,
  consisting of finitely many straight edges and hyperbolic arcs.  A
  working hypothesis (\ref{work:lsing}) in the final section describes
  what would be needed in order to remove the unwanted assumption
  that $\lambda$ does not meet the singular locus $\Lsing$, and
  to prove unconditionally that the Reinhardt trajectory is a smoothed polygon.
\end{remark}

\begin{proof} Fix a Pontryagin extremal trajectory $\lambda$.  By the
  compactness of the interval $[0,t_f]$, it is enough to show that
  there are finitely many switches in a neighborhood of each $t\in
  [0,t_f]$.  By reparametrization, we may assume that $t=0$.

  Here, we give the proof when there are at least two independent
  switching functions $\chi_{ij}$ such that $t=0$ is a limit point
  of the zero set of $\chi_{ij}$.  Theorem
  \ref{lemma:no-chatter2} gives the proof when $t=0$ is a limit
  point of the zero set of only one independent
  switching function.

We define a canonical coordinate system $(\xi_1,\xi_2,\mu_1,\mu_2)$ on
the symplectic manifold $T^*\hstar$ as follows.  Let
\[
\xi_1 = x,\quad \xi_2 =  \frac{3 y^2}{1 + \sqrt{3} x} + \sqrt{3} x = 
\frac{3 (x^2+y^2) + \sqrt{3} x}{1 + \sqrt{3} x}.
\]
A short calculation shows that with respect to these coordinates,
$\hstar$ is given by a semi-infinite rectangle:
\[
-\frac{1}{\sqrt{3}} < \xi_1 < \frac{1}{\sqrt{3}},\quad 1 < \xi_2.
\]
Let $\mu_i:T^*\hstar\to\ring{R}$ be the usual canonical coordinates:
\[
\lambda = \mu_1(\lambda) d\xi_1 + \mu_2(\lambda) d\xi_2,\quad
\lambda\in T^*\hstar.
\]

These canonical coordinates have been chosen to be adapted to the
switching functions:
\begin{align*}
\chi_{13} &= \mu_1 \frac{6 y^3}{1-3x^2 - 3y^2};\\
\chi_{23} &= \mu_2 \frac{12\sqrt{3}y^3}{(1-\sqrt{3} x)(1 + \sqrt{3} x)^2}.
\end{align*}
Thus, up to irrelevant displayed positive factors, we may take $\mu_1$ and
$\mu_2$ to be the switching functions.

Using the rotational symmetries of $U$, we may assume without loss
of generality that $t=0$ is a limit point of the zero sets of both
 $\mu_1$ and $\mu_2$.

 We set $B:= \Lambda - 3 J \lambda_{cost}/2$.  Then Equation
 (\ref{eqn:adjoint}) becomes
\begin{equation}\label{eqn:B'}
B' = \frac{3}{2} \lambda_{cost} [J,X] + [B,X].
\end{equation}
The Hamiltonian expressed in canonical coordinates takes the form
\begin{equation}\label{eqn:ham-xi}
H = g_1 \mu_1 + g_2 \mu_2 + \ang{B,X}
\end{equation}
for some vector field $(g_1,g_2)$ depending on the control $U$.
Recall that $H(\lambda(t),u(t))\equiv 0$ along an extremal $\lambda$.
The adjoint equation for $\mu_i$ is
\begin{align}\label{eqn:mu'}
\begin{split}
\mu_1' &= -\partials{g_1}{\xi_1} \mu_1 - 
\partials{g_2}{\xi_1} \mu_2 -\ang{B,\partials{ X}{ \xi_1}}\\
\mu_2' &= -\partials{g_1}{\xi_2} \mu_1 - 
\partials{g_2}{\xi_2} \mu_2 -\ang{B,\partials{ X}{ \xi_2}}\\
\end{split}
\end{align}
We know that $\mu_1,\mu_2$, and $B$ are absolutely continuous by the
general properties of optimal control.  By the form of the right-hand
side of Equation (\ref{eqn:B'}), we see that $B$ is continuously
differentiable.

We claim that $\mu_i$ are continuously differentiable along a
Pontryagin extremal trajectory.  At issue are the jumps in the
functions ${\partial g_i}/{\partial \xi_j}$ for arbitrary control
functions $u$ on the right-hand side of (\ref{eqn:mu'}).  These
partials derivatives are bounded, so that the form of Equation
(\ref{eqn:mu'}) implies continuity of $\mu_i'$ at $\mu_1=\mu_2=0$.
Near a point where exactly one switching function is zero, the control
is confined to an edge of $U$.  We argue from the form of Equation
(\ref{eqn:mu'}) (or from general facts about switching functions) that
at $\mu_1=0$, $\mu_2=0$, $\chi_{12}=0$ respectively, the right-hand
side of Equation (\ref{eqn:mu'}) does not depend on the control
restricted to the corresponding edge.  This proves the continuity
claim.

We assume that $\mu_1$ and $\mu_2$ have infinitely many zeros 
that accumulate at
$t=0$.  By continuity and Rolle's theorem,
\[
\mu_1(0)=\mu_1'(0)=0,\quad\mu_2(0)=\mu_2'(0)=0.
\]
By Equations (\ref{eqn:ham-xi}) (\ref{eqn:mu'}) and Lemma
\ref{lemma:sl2-basis}, we have $B(0)=0$.

The lifted trajectory is not abnormal, for otherwise 
\[
\lambda_{cost} =
\Lambda(0)=B(0)=\mu_1(0)=\mu_2(0)=0
\]
 contrary to the PMP-projectivity condition.
We set $\lambda_{cost}=-1$.

We claim that $B'(0)=0$. We have
\[
\ang{B'(0),X(0)} = (3/2)\lambda_{cost}\ang{[J,X(0)],X(0)} = 0.
\]
In light of Lemma \ref{lemma:sl2-basis}, to prove the claim, we  assume for a
contradiction that
\begin{equation}\label{eqn:b'-pos}
\epsilon\ang{B'(0),\partials{X}{\xi_i}(0)}>0
\end{equation}
for some $i$ and choice of sign $\epsilon\in\{\pm 1\}$.  By the
mean-value theorem $\mu_j(t) = t\mu_j'(\tau_j) = o(t)$ for some
$\tau_j\in (0,t)$.  By Equation (\ref{eqn:b'-pos}), the forcing term
$\ang{B,\partial X/\partial \xi_i}$ in Equation (\ref{eqn:mu'})
dominates near $t=0$, and we have $\epsilon \mu_i'(t) > C t >0$ for
some $C>0$ and for all sufficiently small $t>0$.  This contradicts our
assumption that $t=0$ is a limit point of the zero set of $\mu_i$.

Note
\[
0=B'(0)=(3/2)\lambda_{cost} [J,X(0)].
\]
So $[J,X(0)]=0$,  which implies $X(0) = J$ so that $x_0=0$ and $y_0=1$.
This completes the proof, except for the missing piece supplied by 
Theorem \ref{lemma:no-chatter2}.
\end{proof}

\begin{theorem}\label{lemma:no-chatter2}
  Let $\lambda:[0,t_f]\to T^*M$ be a Pontryagin extremal that does not
  meet the singular locus $\Lsing$.  Assume that two of the switching
  functions (say $\chi_{12}$ and $\chi_{13}$) only have finitely many zeros
  in some sufficiently small neighborhood of $t=0$.  Then the third
  switching function $\chi_{23}$ also has only finitely many zeros in
  some sufficiently small neighborhood of $t=0$.
\end{theorem}

\begin{proof} 
  We may take the switching function $\chi_{23}$ to be $\nu_2$ (up to
  a positive factor).  We assume for a contradiction that
  $t=0$ is a limit point of the zero
  set of $\nu_2$.  We recall that $\nu_2$ is continuous by the
  PMP.  These observations imply that $\nu_2(0)=0$.

We claim that the control function $u$ takes values in the edge
$U_{23}=\{(0,*,*)\}\subset U$ (up to a set of measure zero).  In fact,
$\chi_{12}<0$ and $\chi_{13}<0$ except at finitely many points in a
small neighborhood of $t=0$.  So $U_{\lambda(t)}\subseteq U_{23}$,
and the claim follows.

We claim that $\nu_2$ is continuously differentiable.  This follows by
examining the ODE it satisfies:
\[
\nu_2' = -f_{1y}\nu_1 - f_{2y} \nu_2 - \ang{B,\partials{X}{y}}.
\]
The term $f_{1y}$ is independent of the control along the edge
$U_{23}$ and is therefore continuous.  The term $f_{2y}$ is bounded
and has jumps only when $\nu_2=0$.  The terms $\nu_1,\nu_2,B,\partial
X/\partial y$ are continuous.  This gives the claim.

We claim that $\nu_2'(0)=0$.  In fact, by Rolle, $\nu_2'(t)$ has
infinitely many zeros that accumulate at $0$.  The claim follows.

We define $\lambda^{(i)}(\lambda_0,t)$ to be the lifted trajectory
with constant control $u=e_i\in U_{23}$ and initial condition
$\lambda^{(i)}(\lambda_0,0)=\lambda_0\in T^*M$, for $i=2,3$.  The
lifted trajectories $\lambda^{(i)}(\lambda_0,t)$ are real analytic in
$\lambda_0$ and $t$.  Let $\nu^{(i)}_2(\lambda_0,t)$ be the
$\nu_2$-component of $\lambda^{(i)}(\lambda_0,t)$.  We have a leading
term
\[
\nu^{(i)}_2(\lambda_0,t) = t^d a_d + O(t^{d+1}),\quad
d = d^{(i)}(\lambda_0),\quad a_d = a_d^{(i)}(\lambda_0)\ne0.
\]
We restrict to parameters $\lambda_0$ near $\lambda(0)$ such that
$\nu_2^{(i)}(\lambda_0,0)=0$ so that $d^{(i)}(\lambda_0)>0$.

We claim that $d^{(i)}(\lambda_0) = d(\lambda_0)$ and
$a_d^{(i)}(\lambda_0)=a_d(\lambda_0)$ are independent of $i$ and that
$1 \le d(\lambda_0) \le 3$.  To prove the claim, we compute the power
series expansion of $\nu_2^{(i)}(\lambda_0,t)$ at $t=0$ using the
explicit solutions to the ODE; and we compare the coefficients for
$i=2,3$.  Explicit formulas are found in the computer code.  When
$d^{(i)}(\lambda_0)> 2$, we compute that
$a_3^{(i)}(\lambda_0)=-1\ne0$, which is independent of both $i$ and
$\lambda_0$.  In particular $d\le 3$.  This proves the claim.

The theorem follows more or less from this last claim.  By the
Weierstrass preparation theorem, the zero set of
$\nu_2^{(i)}(\lambda_0,t)$ coincides with that of a polynomial of
degree $d(\lambda(0))$ in $t$ for all small $t$ and $\lambda_0$ in a
neighborhood of $\lambda(0)$. The idea is that switching function
$\nu_2$ is closely approximated by both of the analytic functions
$\nu_2^{(i)}$, for $i=2,3$, so that $\nu_2$ can have at most $d\le 3$
zeros near $t=0$.

Note that $d(\lambda_0)\le d(\lambda(0))$ for
$\lambda_0$ near $\lambda(0)$.  

Assume first that $d(\lambda(0))=1$.  Because we are restricting to
parameters $\lambda_0$ such that $d(\lambda_0)>0$, we have
$d(\lambda_0)=1$ for all $\lambda_0$ near $\lambda(0)$.  The
paths $\nu_2^{(i)}(\lambda_0,t)$ meet the switching
hypersurface transversely.  The continuous differentiability of $\nu_2$
implies a single switch from control $u=e_i$ to $u=e_j$ at the
switching hypersurface.

In the remaining case, $d(\lambda(0))\in\{ 2,3\}$.  Pick small $t_0$
that is not a switching time.  Then $U_{\lambda(t_0)}=\{e_i\}$ for some
$i\in\{2,3\}$.  Then $\nu_2(t_0) = \nu^{(i)}_2(\lambda_0,t_0)$ for some
$\lambda_0$ near $\lambda(0)$, where $d(\lambda_0)\le d(\lambda(0))$.
Because $d\le 3$, the Weierstrass polynomials for
$\nu^{(i)}_2(\lambda_0,t)$ at $t=0$ have at most one zero $t_1$ of
multiplicity greater than $1$.  The time $t_1$, when it exists,
is independent of $i$.  Thus, every time $t\ne t_1$ lies on a
semi-infinite interval $(t_1,\infty)$ or $(-\infty,t_1)$ on which
$\nu_2$ meets the switching surface transversely, with isolated
switchings between controls $u=e_2$ and $u=e_3$ before leaving the
small neighborhood of  $t=0$.  This implies that $\nu_2$ does not
have a limit point at  $t=0$.
\end{proof}

\section{Discussion of proposed endgames}

In this section we offer some speculations about how the proof of the
Reinhardt conjecture might be completed.

\subsection{smoothed polygons}

We have proved that a Pontryagin extremal trajectory $\lambda$ that
does not meet the singular set $\Lsing$ gives a smoothed polygon
$D(g_\lambda)$.  Suppose that $\lambda$ meets $\Lsing$.  We have
proved that $\lambda$ does not remain in $\Lsing$ for any time
interval and that the only way to approach $\Lsing$ is through
chattering.  These are very restrictive conditions.

We suggest a working hypothesis that would complete the proof that the
solution to the Reinhardt problem is a smoothed polygon.  Our
restrictive conditions reduce the analysis to a small neighborhood of
a single point $\lsing$ in the cotangent space

If a lifted trajectory $\lambda$ meets $\Lsing$, we may assume that the
meeting occurs at $t=0$ and that $\lambda(t)\not\in\Lsing$ for some
sufficiently small time interval $t\in(0,t_0]$.  To be concrete, we
may assume after applying an affine transformation that
$\lambda(0)=\lsing\in\Lsing$.  Then the lifted trajectory on this interval 
has a bang-bang control with infinitely many switching times
\begin{equation}
t_1 > t_2 > \cdots > t_k \cdots > 0,
\end{equation}
where $t_0\ge t_1$ and $\lim_{k\mapsto\infty} t_k = 0$.  If we could show that such a
trajectory is not globally optimal among trajectories in $M$ with the
same endpoints, then chattering is nonoptimal, and we would conclude
established that the solution to the Reinhardt problem is a smoothed
polygon.

\begin{working}\label{work:lsing} 
  Let $\lambda$ be a chattering extremal trajectory with
  bang-bang control starting at $\lambda(0)=\lsing$ as just described.
  Then there exists $t^*\in (0,t_0)$ and a competing lifted trajectory
  $\lambda^*$ on $[0,t^*]$ with lower cost
\[
\op{cost}(\lambda^*)< \op{cost}(\lambda)
\]
over the interval $[0,t^*]$, 
and having the same endpoints in $M$ as $\lambda$:
\begin{align*}
(g_\lambda(0),z_\lambda(0))
&=(g_{\lambda^*}(0),z_{\lambda^*}(0))=(I,i)\in \SL\times\h,\\
(g_\lambda(t^*),z_\lambda(t^*))
&=(g_{\lambda^*}(t^*),z_{\lambda^*}(t^*)).
\end{align*}
\end{working}

To prove this working hypothesis, various standard methods for the
treatment of chattering controls might be helpful: blowing-up along
the singular locus, the Poincar\'e map, scaling, and
self-similarity.  See \cite{zelikin2012theory}.

\subsection{a neighborhood of the circle}

We have constructed extremal trajectories with bang-bang controls
that have an arbitrarily large
number of switches.  Each neighborhood of $V$ of $\Lsing$ contains
all but finitely many of these extremal trajectories.  We expect that 
extremal lifted trajectories $\lambda$
that remain close to $\Lsing$ to give hexagonally symmetric
disks $D(g_\lambda)$ that are approximately circles.  In particular,
they should have cost higher than that of the smoothed octagon.

\begin{working}  There exists (an explicit) neighborhood $V=V_{sing}$ of $\Lsing$ and
a natural number $N_V$ such that any extremal trajectory that
does not meet $V$ has at most $N_V$ switches.  Every extremal trajectory
meeting $V$ has cost greater than that of $\lambda_{oct}$.
\end{working}

\subsection{heuristics near the boundary of $\hstar$}

Recall that $\hstar$ is the open subset
of the upper-half plane that satisfies the star inequalities.
We present some heuristics that suggest that trajectories that come
close to the boundary of $\hstar$ necessarily give hexagonally symmetric disks
whose area is greater than that of the smoothed octagon.  We formulate
this as a working hypothesis.  

\begin{working}\label{hyp:star}   There exists (an explicit) neighborhood $V=V_\star$ of the
  boundary of $\hstar$ such that if (the projection $z_\lambda$ to
  $\hstar$ of) an extremal admissible trajectory $\lambda$ meets $V$,
  then the cost of $z_\lambda$ is greater than the cost of the
  smoothed octagon.
\end{working}

\begin{remark} The point of the working hypothesis is to allow us to
  replace $\hstar$ with a slightly smaller compact set
  $K=\hstar\setminus V$.  The star inequalities (Remark
  \ref{rem:star}) are strict because when equality is obtained the
  smallest centrally symmetric hexagon $H$ containing $D$ degenerates
  to a parallelogram.  However, a parallelogram is the smallest
  centrally symmetric hexagon $H$ containing $D$ only if $D=H$ is
  itself a parallelogram of area $\sqrt{12}$.  This suggests that
  extremal trajectories that come sufficiently close to the boundary
  of $\hstar$ are approximately parallelograms of approximate area
  $\sqrt{12}$.  Such $D$ are far from optimal.  By making this
  intuition rigorous, a proof of the working
  hypothesis might be obtained.
\end{remark}

\subsection{direct computer search}

We suggest two different ways that the Reinhardt conjecture might be
completed from here: direct computer search or geometric methods.
We can summarize the previous subsections by saying that we expect
the only interesting lifted trajectories in the cotangent space to pass
\begin{enumerate}
\item near the boundary of $\hstar$ (where $D$ is a near parallelogram);
\item near $\Lsing$ (where $D$ is a near circular disk, including
  smoothed polygons $D_{6k\pm 2}$ with many sides);
\item near $\lambda_{oct}$ (where $D$ is a smoothed octagon).
\end{enumerate}
Let us assume that we have a version of Theorem \ref{thm:local-min} that gives
an explicit neighborhood $V_{oct}$ of $\lambda_{oct}(0)$ on which the
local optimality of $\lambda_{oct}$ holds.

Each Pontryagin extremal trajectory is determined by an initial
condition in $\ring{R}_{cost}\times T_I^*M$.  It is convenient to
consider the projectivized variant:
\[
\ring{P}(\ring{R}_{cost}\times T_I^* M^\star),
\]
This is an explicit $7$-dimensional manifold.  It can be
reduced by two dimensions to a $5$-dimensional manifold
by the vanishing of the maximized
Hamiltonian (\ref{sec:pmp}) and setting the start time $t=0$ at a
switching time between controls $u=e_3$ and $u=e_2$, which
gives $\nu_2^0=0$.

We might try to make a direct computer search through this space (say
using interval arithmetic) and show that there is nothing better than
the smoothed octagon.  We might find for example by explicit search
that the smoothed polygons of Section \ref{sec:polygon} are the only
Pontryagin extremal trajectories (away from the singular arc).

Using our working hypotheses, by excluding a neighborhood $V_\star$ of the
boundary of $\hstar$, a neighborhood $V_{oct}$ of $\lambda_{oct}$, and a
neighborhood $V_{sing}$ of $\Lsing$, we expect numerically stable lifted trajectories
with a uniformly bounded number of switches.  Given an initial
condition $\lambda_0$ in the $5$-dimensional manifold, we extend the
trajectory until it enters one of these excluded neighborhoods $V$ (in
which case we reject $\lambda_0$), until $\lambda$ meets the terminal
conditions (in which case we compare the trajectory's cost to
$\lambda_{oct}$), or until $t\ge \pi/3$ (in which case we reject the
trajectory it for having higher cost than the circle by Equation \ref{eqn:cost-min}).

\subsection{geometric methods}

The transversality conditions of PMP imply that a Pontryagin extremal
lifted trajectory is a closed loop $\lambda$ in 
\[
\ring{P}(\ring{R}_{cost} \times T^*M \setminus \Lsing)/\ang{R}.
\]
  (We remove a neighborhood of the singular locus $\Lsing$.)

We can consider an optimization over each homology class.
We have homotopy group $\pi_1(\SL)=\ring{Z}$ and the canonical map
$\pi_1(\SL)\to \pi_1(\SL/R)=\ring{Z}$ is multiplication by $3$.  The
Reinhardt lifted trajectory gives a generator of $\pi_1(\SL/R)$.  We
may restrict to such trajectories.

We might try to adapt the arguments of \cite[Chapter
17]{agrachev2013control}: the Poincar\'e-Cartan integral invariant,
Hamilton-Jacobi-Bellman, etc.

\bibliography{refs} 

\begin{thebibliography}{Naz88}

\bibitem[AS13]{agrachev2013control}
Andrei~A Agrachev and Yuri Sachkov.
\newblock {\em Control theory from the geometric viewpoint}, volume~87.
\newblock Springer Science \& Business Media, 2013.

\bibitem[Bae14]{baez}
John Baez.
\newblock Packing smoothed octagons, 11 2014.
\newblock
  \url{http://blogs.ams.org/visualinsight/2014/11/01/packing-smoothed-octagons/}.

\bibitem[BE14]{baez-egan}
John Baez and Greg Egan.
\newblock A packing pessimization problem, 9 2014.
\newblock
  \url{https://golem.ph.utexas.edu/category/2014/09/a_packing_pessimization_proble.html}.

\bibitem[Hal11]{hales2011reinhardt}
Thomas~C Hales.
\newblock On the {Reinhardt} conjecture.
\newblock {\em Vietnam Journal of Mathematics}, 39(3):287--307, 2011.
\newblock arXiv:1103.4518.

\bibitem[Lib12]{liberzon2012calculus}
Daniel Liberzon.
\newblock {\em Calculus of variations and optimal control theory: a concise
  introduction}.
\newblock Princeton University Press, 2012.

\bibitem[Mit98]{mittenhuber1998dubins}
Dirk Mittenhuber.
\newblock {Dubins}’ problem in hyperbolic space.
\newblock {\em Geometric Control and Non-Holonomic Mechanics}, 25:101--114,
  1998.

\bibitem[MP98]{monroy1998non}
Felipe Monroy-P{\'e}rez.
\newblock Non-euclidean {Dubins}' problem.
\newblock {\em Journal of dynamical and control systems}, 4(2):249--272, 1998.

\bibitem[Naz88]{nazarov1988reinhardt}
FL~Nazarov.
\newblock Reinhardt's problem of lattice packings of convex domains: Local
  extremality of the reinhardt octagon.
\newblock {\em Journal of Mathematical Sciences}, 43(5):2687--2693, 1988.

\bibitem[Rei34]{Reinhardt:1934}
K.~Reinhardt.
\newblock \"{U}ber die dichteste gitterf\"ormige {L}agerung kongruenter
  {B}ereiche in der {E}bene und eine besondere {A}rt konvexer {K}urven.
\newblock {\em Abh. Math. Sem., Hamburg, Hansischer Univ., Hamburg},
  10:216--230, 1934.

\bibitem[ZB12]{zelikin2012theory}
Michail~I Zelikin and Vladimir~F Borisov.
\newblock {\em Theory of chattering control: with applications to astronautics,
  robotics, economics, and engineering}.
\newblock Springer Science \& Business Media, 2012.

\end{thebibliography}
\bibliographystyle{alpha}

\end{document}